\documentclass[10pt]{amsart}

\usepackage{amsfonts}
\usepackage{amssymb}
\usepackage{color}
\usepackage{amsmath, amsthm, latexsym}
\usepackage{mathtools}
\usepackage[all]{xy}
\usepackage{tikz}
\usetikzlibrary{fit,matrix}
\usepackage{pgfplots}
\usepgfplotslibrary{fillbetween}
\usepackage{hyperref}
\usepackage[margin=1in]{geometry}

\def \C{\mathbb{C}}
\def \Z{\mathbb{Z}}
\def \R{\mathbb{R}}

\def \Q{\mathbb{Q}}
\def \A{\mathbb{A}}

\def \V{\mathcal{V}}
\def \K{\mathcal{K}}
\def \O{\mathcal{O}}

\def \Tm{\mathcal{T}}

\def \ee{\mathbf{e}}
\def \uu{\mathbf{u}}
\def \vv{\mathbf{v}}
\def \ww{\mathbf{w}}
\def \xx{\mathbf{x}}
\def \yy{\mathbf{y}}

\def \trop{\operatorname{trop}}
\def \strop{\operatorname{strop}}
\def \Log{\operatorname{Log}}
\def \sLog{\operatorname{sLog}}

\def \ord{\operatorname{ord}}

\def \Hom{\operatorname{Hom}}
\def \SL{\operatorname{SL}}
\def \GL{\operatorname{GL}}

\def \SU{\operatorname{SU}}
\def \stab{\operatorname{Stab}}

\def \Div{\operatorname{div}}
\def \ker{\operatorname{ker}}
\def \diag{\operatorname{diag}}
\def \Sym{\operatorname{Sym}}

\newcommand*\diff{\mathop{}\!\mathrm{d}}

\newcommand{\set}[1]{\rleft\{ {#1} \rright\}}
\newcommand{\with}{\,\colon\,}

\newcommand{\rleft}{\mathopen{}\mathclose\bgroup\left}
\newcommand{\rright}{\aftergroup\egroup\right}

\theoremstyle{plain}
\newtheorem{Th}{Theorem}[section]
\newtheorem{Lem}[Th]{Lemma}
\newtheorem{Prop}[Th]{Proposition}
\newtheorem{Cor}[Th]{Corollary}

\newtheorem{Question}[Th]{Question}

\newtheorem{THM}{Theorem}
\newtheorem{CORO}[THM]{Corollary}
\newtheorem{CONJEC}[THM]{Conjecture}

\theoremstyle{definition}
\newtheorem{Ex}[Th]{Example}
\newtheorem{Def}[Th]{Definition}
\newtheorem{Rem}[Th]{Remark}

\begin{document}

\title{Spherical amoebae and a spherical logarithm map}
\author{Victor Batyrev} 
\author{Megumi Harada}
\author{Johannes Hofscheier}
\author{Kiumars Kaveh}
\date{\today}

\begin{abstract}
Let $G$ be a connected reductive algebraic group over $\C$ with a maximal compact subgroup $K$. Let $G/H$ be a (quasi-affine) spherical homogeneous space. In the first part of the paper, following Akhiezer's definition of spherical functions, we introduce a $K$-invariant map 
$\sLog_{\Gamma, t}: G/H \to \R^s$ which depends on a choice of a finite set $\Gamma$ of dominant weights and $s = |\Gamma|$.
We call $\sLog_{\Gamma, t}$ a \textbf{spherical logarithm map}. We show that when $\Gamma$ generates the highest weight monoid of $G/H$, the image of the spherical logarithm map parametrizes $K$-orbits in $G/H$. This idea of using the spherical functions to understand the geometry of the space $K \backslash G/H$ of $K$-orbits in $G/H$ can be viewed as a generalization of the classical Cartan decomposition. In the second part of the paper, we define the \textbf{spherical amoeba} (depending on $\Gamma$ and $t$) of a subvariety $Y$ of $G/H$ as $\sLog_{\Gamma, t}(Y)$, and we ask for conditions under which the image of a subvariety $Y \subset G/H$ under $\sLog_{\Gamma, t}$ converges, as $t \to 0$, in the sense of Kuratowski to its \textbf{spherical tropicalization} as defined by Tevelev and Vogiannou. We prove a partial result toward answering this question, which shows in particular that the valuation cone is always contained in the Kuratowski limit of the spherical amoebae of $G/H$. We also show that the limit of the spherical amoebae of $G/H$ is equal to its valuation cone in a number of interesting examples, including when $G/H$ is horospherical, and in the case when $G/H$ is the space of hyperbolic triangles.   
\end{abstract}

\maketitle

\bigskip
{\bf Comments are welcomed.}

\tableofcontents

\section*{Introduction}
Let $Y \subset (\C^*)^n$ be an algebraic subvariety of the $n$-dimensional complex algebraic torus. 
The (classical) \emph{amoebae} $\mathcal{A}_t(Y)$ of $Y$, for a varying real parameter $t$, have been much studied in the past few decades in  several contexts, including in tropical and non-Archimedean geometry. The amoeba $\mathcal{A}_t(Y)$ is defined to be the image of $Y$ under the logarithm map $\Log_t: (\C^*)^n \to \R^n$ defined for fixed $t > 0$ by 
$$\Log_t(x_1, \ldots, x_n) := (\log_t|x_1|, \ldots, \log_t|x_n|).$$
Note that the map $\Log_t$ is invariant under the action of the \emph{compact} torus $(S^1)^n \subset (\C^*)^n$. In fact, with respect to the isomorphism $(x_1, \ldots, x_n) \mapsto ((\theta_1, \ldots, \theta_n), (r_1, \ldots, r_n))$ where $x_i = r_i e^{i\theta_i}$ between $(\C^*)^n$ and $(S^1)^n \times \R_{>0}^n$, the logarithm map $\Log_t$ can be seen to be the projection onto the second component(s), followed by the usual logarithm function $r_i \mapsto \log_t r_i$. 

For an algebraic subvariety $Y$ in $(\C^*)^n$, it is well-known that the amoeba $\mathcal{A}_t(Y)$ stretches to infinity in several directions, also known as its {\it tentacles}. These tentacles indicate the directions in which $Y$ approaches infinity, or in other words, they parametrize the orbits at infinity in its tropical compactification. The precise statement is as follows (cf. \cite{Jonsson}): as $t \to 0$, $\mathcal{A}_t(Y)$ converges to the tropical variety $\trop(Y)$ as subsets in $\R^n$, where convergence is defined here in the sense of Kuratowksi (see Definition~\ref{def-Kuratowski}). The theory of amoebae and tropicalization connects classical complex geometry with tropical geometry, which studies piecewise-linear objects in Euclidean space, and the study of these interactions is an active research area. 

In this paper, we take initial steps toward generalizing the above to the \emph{non-abelian} setting. More precisely, we choose to view the complex torus $(\C^*)^n$ as a special case of a \textit{spherical homogeneous space} in which the group is abelian. {We then see whether, or how, the abelian story (as briefly outlined above) might be generalized for non-abelian spherical homogeneous spaces. Before further details, we first briefly recall the setting. } Let $G$ be a complex connected reductive algebraic group. A homogeneous space $G/H$ of $G$ is called {\it spherical} if a Borel subgroup $B$ of $G$ has a dense orbit, with respect to the usual action by left multiplication. (See Section~\ref{sec-prelim} for details.) Thus, the complex torus $(\C^*)^n$ is the special case in which $G = (\C^*)^n$ is the torus itself, $H = \{1\}$ is the trivial subgroup, and the Borel subgroup $B$ is again the torus $(\C^*)^n$ itself. 
 
{One of our fundamental motivations for this paper is the desire to better understand the $K$-orbit space of a spherical homogeneous space $G/H$ as above, where $K$ is a maximal compact subgroup of $G$. Indeed, in the abelian case, the logarithm map $\Log_t$ can be viewed as the quotient map of $(\C^*)^n$ by its maximal compact subgroup $(S^1)^n$, and we can view $\R^n_{>0}$ as the parameter space of the $(S^1)^n$-orbits in $(\C^*)^n$. Similarly we may ask whether there exists a non-abelian, or spherical, analogue of the classical logarithm map on $G/H$ which allows us to naturally parametrize the $K$-orbit space $K \backslash G/H$ of $G/H$. This problem can be viewed as a generalization of the classical Cartan and polar decompositions in Lie theory.  However, as far as we know, there is no general and concrete description of the $K$-orbit space $K\backslash G/H$ known. (For some important results in this direction see \cite{Knop et al}.) A main motivation for the present paper is the belief that there should be a natural isomorphism between $K \backslash G/H$ and the real closure of the \textbf{valuation cone} $\V_{G/H}$ (see Section \ref{subsec-background} for the definition) of an arbitrary spherical variety $G/H$ (over $\C$). In the abelian case, the valuation cone is precisely $\R^n$, so the logarithm $\Log_t$ provides this natural isomorphism.  We think of this paper as taking some steps forward in the search of such a natural isomorphism (see also Conjecture \ref{conj-Batyrev}). (We also note that related ideas have appeared in \cite{Feu}.) However, it should be noted that the results we obtain are of a different nature than a direct generalization of the torus case. In contrast to the abelian situation in which $\Log_t$ is an isomorphism $K \backslash G/H \cong \overline{\V}_{G/H}$, what we see in our results is that, \emph{in the limit as $t \to 0$}, the image of our spherical logarithm map $\sLog_t$ (to be defined precisely below) converges in an appropriate sense to the valuation cone. In particular, for fixed $t$, our $\sLog_t$ does not provide the desired canonical isomorphism.   } 
 
 In addition to the above, we were motivated by the notion of a \textbf{spherical tropicalization map} associated to a spherical homogeneous space $G/H$. Spherical tropical geometry was introduced  and developed by Tevelev and Vogiannou \cite{Vogiannou} and by the fourth author and Manon \cite{Kaveh-Manon}. 
We believe that there should be a theory of spherical amoebae for $G/H$ which correspond to the classical amoebae in the abelian case, such that the ``spherical amoebae converge to the spherical tropicalization'' in an appropriate sense (see also Remark \ref{rem-KM}).

With these motivations in mind, we now describe more precisely the contributions of this article. Let $G/H$ be a spherical homogeneous space as above. For technical reasons we additionally assume that $G/H$ is quasi-affine.  (This is not terribly restrictive, as we explain in Remark~\ref{remark: not severe}.) 
Let $\Lambda^+_{G/H}$ denote the semigroup of $B$-highest-weights of irreducible $G$-representations appearing in the coordinate ring $\C[G/H]$, considered as a $G$-module in the natural way.
Following some rather little-known work of Akhiezer \cite{Akhiezer-rk}, we define a $K$-invariant \textbf{spherical function} $\phi_\lambda$ on $G/H$ associated to a $B$-weight $\lambda$ for $\lambda \in \Lambda^+_{G/H}$. We normalize these spherical functions by setting $\phi_\lambda(eH) = 1$.  
Our first observations concerning Akhiezer's spherical functions is the following; see Propositions~\ref{prop: algebra generators} and~\ref{prop: separate K orbits} for precise statements. 

\begin{THM}   \label{th-intro:separate K-orbits and algebra generators}
Suppose $\lambda_1, \ldots, \lambda_s \in \Lambda_{G/H}^+$ generate $\Lambda^+_{G/H}$ as a semigroup. 
Then the corresponding spherical functions $\phi_{\lambda_1}, \ldots, \phi_{\lambda_s}$ generate the algebra of $K$-invariant functions on $G/H$. It follows that the image of the spherical logarithm map corresponding the choice of $\{\lambda_1, \ldots, \lambda_s\}$, parametrizes the $K$-orbit space $K \backslash G / H$.
Moreover, the spherical functions separate $K$-orbits. More precisely, for any two distinct (hence disjoint) orbits $K$-orbits $O_1$ and $O_2$ in $G/H$, there exists some $\lambda \in \Lambda^+_{G/H}$ such that $\phi_\lambda(O_1) \neq \phi_\lambda(O_2)$. 
\end{THM} 

Motivated by the above, let 
$\Gamma = \{\lambda_1, \ldots, \lambda_s\} \subset \Lambda^+_{G/H}$. We define $\Phi_\Gamma: G/H \to \R_{\geq 0}^s$ by: 
$$\Phi_\Gamma(x) = (\phi_{\lambda_1}(x), \ldots, \phi_{\lambda_s}(x))$$
(details are in Section~\ref{sec: spherical functions}). For a fixed $t>0$, we then define the \textbf{spherical logarithm map} (depending on $\Gamma$ and with base $t$) to be the map $\sLog_{\Gamma, t} := \log_t(\Phi_\Gamma): G/H \to \R^s$ (see Definition~\ref{def-Phi}). 
\footnote{Note that the spherical tropicalization map defined by the fourth author and Manon in \cite{Kaveh-Manon}
naturally takes values in a $\Q$-vector space. However, since our spherical logarithm map uses a logarithm, we need to work with $\R$ coefficients.} 
If we assume that the above subset $\Gamma$ generates $\Lambda^+_{G/H}$ as a semigroup, then from Theorem~\ref{th-intro:separate K-orbits and algebra generators} it follows that we obtain a parametrization of $K \backslash G/H$, i.e., the points in the image $\sLog_{\Gamma, t}(G/H)$ parameterize $K$-orbits in $G/H$. This is a positive step in the broader program of understanding $K \backslash G/H$. We will have more to say about the relation between the image of $\sLog_{\Gamma,t}$ and the valuation cone $\V_{G/H}$ below.

The spherical functions used to define $\sLog_{\Gamma,t}$ have somewhat subtle properties. The next result, below, gives an inequality involving a product of two spherical functions. The essential point here is that spherical functions are \emph{not} necessarily multiplicative: given two $B$-weights $\lambda$ and $\mu$, in general we can have $\phi_{\lambda} \phi_{\mu} \neq \phi_{\lambda+\mu}$ (see Example~\ref{ex-sph-functions-not-multiplicative}). However, it is worth emphasizing that when $G/H$ is horospherical the multiplicativity indeed holds (Lemma \ref{lem-multiplication and tails}).  When $G/H$ is not horospherical, our next result gives some information about the relation between the product $\phi_\lambda \phi_\mu$ and the other spherical functions. To state it, 
we need the following notation. For a $B$-weight $\lambda$, let $V_\lambda$ denote the irreducible $G$-representation with highest weight $\lambda$. Given a $G$-module $W$ we say $V_\lambda$ appears in $W$ if the decomposition of $W$ into irreducible $G$-representations contains a copy of $V_\lambda$. Here we view $\C[G/H]$ as a $G$-module. Given $V_\lambda$ and $V_\mu$ both appearing in $\C[G/H]$, we let $\langle V_\lambda \cdot V_\mu \rangle$ denote the $\C$-span of the set of products $f \cdot g \in \C[G/H]$ where $f \in V_\lambda$ and $g \in V_\mu$. We have the following; for a precise statement see Proposition
\ref{prop-phi-lambda-inequalities}. 

\begin{THM} \label{intro-prop-phi-lambda-inequalities}
Suppose $V_\lambda$ and $V_\mu$ appear in $\C[G/H]$ and suppose $V_\gamma$ appears in $\langle V_\lambda \cdot V_\mu \rangle  \subseteq \C[G/H]$. Then there exists a constant $c>0$ such that 
for all $x \in G/H$ we have 
$$\phi_\lambda (x) \phi_\mu(x) \geq c \phi_\gamma(x).$$
In particular, since $V_{\lambda+\mu}$ always appears in $\langle V_\lambda \cdot V_\mu \rangle$, we see that there is a constant $c>0$ such that $\phi_\lambda(x) \phi_\mu(x) \geq c \phi_{\lambda+\mu}(x)$ for all $x \in G/H$. 
\end{THM}

As we have observed, our definition of $\sLog_{\Gamma,t}$ depends on the choice of $\Gamma$. In the special case when $G/H$ is horospherical and thus the spherical functions are multiplicative, we can define the spherical logarithm map $\sLog_t$ \emph{canonically} (i.e. independently of the choice of $\Gamma \subset \Lambda^+_{G/H}$) by 
\[
\sLog_{\Gamma, t}(x) := \log_t(\Phi(x)) \, \, \textup{ for } x \in G/H
\]
where $\Phi(x)$ is defined by $\langle \Phi(x), \lambda \rangle = \phi_\lambda(x)$, $\forall \lambda \in  \Lambda^+_{G/H}$. 
 In this case, the image of the map $\sLog_t$ lands naturally in the $r$-dimensional vector space $\mathcal{Q}_{G/H} \otimes \R := \Hom(\Lambda^+_{G/H}, \R)$, where $r$ is the rank of the spherical variety $G/H$. This is also the vector space in which the valuation cone $\V_{G/H}$ naturally lies (see Section~\ref{sec-prelim}, as well as Remark \ref{rem-slog-horosph}). 
As we mentioned above, this paper was partly motivated by the search for a natural isomorphism between $K \backslash G/H$ and the valuation cone. Since $\sLog_{\Gamma,t}(G/H)$ parametrizes $K \backslash G/H$, it is natural to ask for a relationship between it and $\V_{G/H}$. A choice of generating set $\Gamma \subset \Lambda^+_{G/H}$ gives a linear embedding of the valuation cone $\V_{G/H}$ into $\R^s$, the codomain of $\sLog_{\Gamma,t}$ in the general case, given by:
$$v \mapsto (\langle v, \lambda_1\rangle, \ldots, \langle v, \lambda_s\rangle), \quad v \in \V_{G/H}.$$ 
When comparing $\V_{G/H}$ with $\sLog_{\Gamma,t}(G/H)$ in what follows, we identify $\V_{G/H}$ with its image in $\R^s$ under the above embedding. 

We now turn to the related question of spherical amoebae.   Let $Y \subset G/H$ be a subvariety of $G/H$. We define \textbf{the spherical amoeba of $Y$} to be the image $\sLog_{\Gamma,t}(Y)$ of $Y$ under the spherical logarithm. One of our motivations for this manuscript was to ask: do the spherical amoebae approach (in a suitable sense) the spherical tropicalization of $Y$ as $t \to 0$? Of course, as in the classical (abelian) case, one needs to be precise about what the word ``approach'' means. Following \cite{Jonsson}, here we take it to mean the convergence of subsets \emph{in the sense of Kuratowski} (Definition~\ref{def-Kuratowski}). Moreover, in the non-abelian case, the image of the spherical homogeneous space $G/H$ under the spherical tropicalization map is not necessarily the entire codomain; instead, the image is the valuation cone $\V_{G/H}$. This is in contrast to the classical case $G=(\C^*)^n$, in which 
the image is the entire vector space $\R^n$. We ask the following (cf. Question~\ref{question: sph-amoeba-approach-sph-trop}): 
\begin{equation*}\label{Convergence of amoebae}
\begin{minipage}{0.7\linewidth}
\emph{(1) Under what conditions do the images $\sLog_{\Gamma,t}(G/H)$ approach the valuation cone $\V_{G/H}$ as $t \to 0$ in the sense of Kuratowski convergence of subsets? }
\emph{(2) For $Y \subseteq G/H$ a subvariety, under what conditions do the spherical amoebae $\sLog_{\Gamma,t}(Y)$ approach $\strop(Y)$ as $t \to 0$ in the sense of Kuratowski convergence of subsets?}
\end{minipage}
\end{equation*}

{In Remark \ref{rem-image-sph-log-conv-val-cone} we sketch a proof strategy to give a positive answer to Part (1) of the above question; moreover, as we discuss further below, in Section~\ref{sec-exs} we show that in many examples, the Kuratowski limit of $\sLog_{\Gamma,t}(G/H)$ is indeed
the valuation cone $\V_{G/H}$. We expect that a proof of Part (2) to be more difficult, and (as in the torus case \cite{Jonsson}), an answer might require tools from non-Archimedean geometry.   Nevertheless we conjecture the following. 

\begin{CONJEC}
If $G/H$ is horospherical, and $Y \subseteq G/H$ is a subvariety, then the statement in the Question (2) (as stated above) holds, i.e., the spherical amoebae $\sLog_{\Gamma, t}(Y)$ approaches $\strop(Y)$ in the sense of Kuratowski as $t \to 0$. 
\end{CONJEC} 
}

We also note that progress on (1) above would partially answer the broad motivating question of finding a relation between $K \backslash G/H$ and $\V_{G/H}$.  In this paper, 
we prove some preliminary results which address the question (1) above. Firstly, we answer the ``curve case'', as follows. 
The precise statement is Theorem~\ref{th-phi-approach-inv-valuation}. 

\begin{THM}   \label{th-intro:phi-approach-inv-valuation}
Let $v_\gamma \in \V_{G/H}$ be the $G$-invariant valuation associated to a formal Laurent curve $\gamma$ in $G/H$, convergent for small enough $t$. Let $\lambda \in \Lambda^+_{G/H}$ and let $\phi_\lambda: G/H \to \R_{>0}$ be the corresponding spherical function. Then for any highest weight vector $f_\lambda \in V_\lambda \subset \C[G/H]$ we have
\begin{equation*}  
\lim_{t \to 0} \log_t(\phi_\lambda(\gamma(t))) = 2 v_\gamma (f_\lambda).
\end{equation*}
\end{THM}

Theorem \ref{th-intro:phi-approach-inv-valuation} yields the following corollary (see Corollary~\ref{corollary Kuratowski}).

\begin{CORO}\label{coro: intro} 
(1) The Kuratowski limit, as $t \to 0$, of the image $\sLog_{\Gamma, t}(G/H)$ contains the valuation cone $\V_{G/H}$. 
(2) When $G/H$ is horospherical, this limit is the entire vector space, which in this case coincides with the valuation cone. 
\end{CORO}

Statement (1) in Corollary~\ref{coro: intro} shows that $\V_{G/H}$ is contained in the Kuratowski limit of the images $\sLog_{\Gamma, t}(G/H)$. When is the limit exactly $\V_{G/H}$? 
In Section~\ref{sec-exs}, we give some examples for which, with appropriate choices of $\Gamma$, the images $\sLog_{\Gamma, t}(G/H)$ limits to $\V_{G/H}$ as $t \to 0$: (i) the basic (quasi) affine space $\SL_n(\C)/U$ (cf. Example~\ref{ex-sph-var}(4)), (ii) the ``group case'' with $\SL_n(\C) \times \SL_n(\C)$ acting on $\SL_n(\C)$ (cf. Example~\ref{ex-sph-var}(5)), (iii) $\SL_n(\C)/\SL_{n-1}(\C)$, and (iv) the space of hyperbolic triangles. These examples give some evidence that the answer to the question (1) above may be positive, and hence the $K$-orbits $K \backslash G/H$ is related, in an appropriate limit, to $\V_{G/H}$. In fact, the first author previously conjectured a tight relation between $K \backslash G/H$ and $\V_{G/H}$. Since the conjecture is not recorded elsewhere, we put it here. 

{\begin{CONJEC}[Batyrev]  \label{conj-Batyrev}
The $K$-orbit space $K\backslash G/H$ is a stratified manifold with corners where the boundary faces $X_\sigma$ are in natural bijection with the faces $\sigma$ of the valuation cone $\V_{G/H}$, and the stabilizers of the points in 
the relative interior of $X_\sigma$
are maximal compact subgroups in the corresponding satellite subgroup of $H$ (as defined in \cite{Batyrev-Moreau}). Moreover, $K\backslash G/H$ is homeomorphic to the valuation cone as a stratified space, i.e., there is a homeomorphism $\Phi: \V_{G/H} \to K \backslash G/H$ that, for every face $\sigma \subset \V_{G/H}$, restricts to a homeomorphism between the relative interior of $\sigma$ and that of $X_\sigma$.  
\end{CONJEC}
}

{As a special case we also state the following conjecture.
\begin{CONJEC}  \label{conj-horosph-K-orbit-space}
A spherical homogeneous space $G/H$ is horospherical if and only if the $K$-orbits space $K\backslash G/H$ is homeomorphic to a Euclidean space. 
\end{CONJEC}

\noindent We believe that the spherical logarithm maps $\sLog_{\Gamma, t}$ may provide a tool to attack Conjecture \ref{conj-horosph-K-orbit-space}.

For historical context, we should note here that it has been known for some time that the theory of moment maps from symplectic geometry can be used to parametrize $K$-orbits for multiplicity-free spaces. The theory of multiplicity-free Hamiltonian spaces are related to this question of $K$-orbits on spherical $G/H$ for the following reason: if $X$ is smooth spherical variety embedded in a projective space which is equipped with a linear $K$-action and a $K$-invariant Hermitian structure (and hence has a $K$-invariant K\"ahler structure), then $X$ is a multiplicity-free Hamiltonian $K$-space. Thus, a $K$-equivariant embedding of $G/H$ into a projective space gives $G/H$ such a structure. Moreover, for a multiplicity-free Hamiltonian $K$-space $M$, it is known that the \textbf{Kirwan map}\footnote{ The Kirwan map is the composition of the $K$-moment map $\Phi: M \to \mathfrak{k}^*$ with the quotient map $\mathfrak{k}^*/K \cong \mathfrak{t}^*_+$ where $\mathfrak{t}$ denotes the Lie algebra of the (compact) maximal torus $T$ of $K$.} provides a parametrization of the $K$-orbits in $M$} (see \cite[Proposition 5.1]{Brion-moment}). {Thus, this Kirwan map provides another potential approach to Conjectures~\ref{conj-Batyrev} and~\ref{conj-horosph-K-orbit-space}. 
}
 
We now outline the contents of the paper.  In Section~\ref{sec-prelim} we briefly recall necessary background, including the definition of the spherical tropicalization map as in \cite{Vogiannou, Kaveh-Manon}.  Then in Section~\ref{sec: spherical functions} we define the spherical functions, following Akhiezer, and prove several properties about them. We also define the notion of a spherical logarithm map. In Section~\ref{sec: limit} we give precise statements of the Kuratowski convergence of subsets, and formulate our questions concerning the convergence of spherical amoebae.  We also prove the ``curve case'' recounted above, as well as the horospherical case. {Finally, in Section~\ref{sec-exs} we compute spherical functions in several important examples, and in some examples, we also give some natural parametrizations of the $K$-orbit space $K \backslash G/H$. }

\medskip

\noindent{\bf Acknowledgements.}
We would like to thank Dmitri Timashev for invaluably helpful personal discussions and correspondences clarifying several details. In particular, Example~\ref{ex-sph-functions-not-multiplicative}(1) is due to him. 
We also thank St\'ephanie Cupit-Foutou for helpful correspondence (and in particular for Remark~\ref{rem-Stephanie}). 
Some of this work was conducted at the Fields Institute for Research in the Mathematical Sciences during the fourth author's stay there as a Fields Research Fellow, and we thank the Fields for its support and hospitality. 
The second author was supported in part by a Natural Science and Engineering Research Council of Canada Discovery Grant and a Canada Research Chair (Tier 2) from the Government of Canada. The fourth author was partially supported by a National Science Foundation Grant (DMS-1601303) and a Simons Collaboration Grant for Mathematicians.

\section{Preliminaries on spherical varieties and tropicalization}
\label{sec-prelim}

In this section, we set some notation and briefly review some facts regarding spherical varieties.

\subsection{Background on spherical geometry} 
\label{subsec-background}
Let $G$ be a connected reductive algebraic group over $\C$. Let $B$ denote a choice of Borel subgroup $B$ of $G$ and let $T$ be a maximal torus with $T \subseteq B$. The weight lattice of $T$ is denoted by $\Lambda$, and the semigroup of dominant weights corresponding to the choice of $B$ is correspondingly denoted by $\Lambda^+$. The cone generated by $\Lambda^+$ is the positive Weyl chamber $\Lambda^+_\R$.
For a dominant weight $\lambda \in \Lambda^+$ we denote the irreducible $G$-module with highest weight $\lambda$ by $V_\lambda$. We usually denote a highest weight vector in 
$V_\lambda$ by $v_\lambda$.

Recall that a normal $G$-variety $X$ is called \textbf{spherical}, or a \textbf{spherical ($G$-)variety}, 
if there exists a dense $B$-orbit in $X$.  (Since all the Borel subgroups are conjugate, this condition is  independent of the choice of Borel subgroup $B$).
A homogeneous space ${G/H}$ is called spherical if it is spherical with respect to the left action of $G$ on itself by multiplication.

Let $X$ be a spherical $G$-variety and let $\C(X)$ be the field of rational functions on $X$. Since $G$ acts on $X$, there is also a natural $G$-action on $\C(X)$. We say that $f \in \C(X)$ is a $B$-eigenfunction (sometimes also called $B$-semi-invariants) if it has the property that $b \cdot f = e^{\lambda}(b) f$ for some $B$-weight $\lambda$ where $e^\lambda$ denotes the character associated to $\lambda$, i.e., for all $x \in X$ we have $(b \cdot f)(x) = f(b^{-1}x) = e^{\lambda}(b) f(x)$. Let $\C(X)^{(B)}$ denote the set of $B$-eigenfunctions in $\C(X)$, and let $\Lambda_X$ denote the subset of $\Lambda$ consisting of all weights arising from $B$-eigenfunctions, namely, 
\[
\Lambda_X := \{ \lambda \, \mid \, \exists \textup{ $B$-eigenfunction } f \in \C(X) \textup{ with } b \cdot f = e^{\lambda}(b) f \} \text{.} 
\]
 Consider the map $\C(X)^{(B)} \to \Lambda$ associating to a $B$-eigenfunction its corresponding weight (in the notation above, the map sends $f$ to $\lambda$). It is clear that this is a group homomorphism with respect to multiplication in $\C(X)^{(B)}$ and addition in $\Lambda$ and that its image is $\Lambda_X$ by definition. Hence $\Lambda_X$ is a sublattice of $\Lambda$. Moreover, since $X$ is spherical and thus there exists an open dense $B$-orbit, it follows that this map has kernel precisely the (non-zero) constant functions $\cong \C^*$, thus inducing an  
isomorphism between $\C(X)^{(B)} / \C^*$ and $\Lambda_X$. 

We briefly recall some standard examples of spherical varieties. 
\begin{Ex} \label{ex-sph-var}
\begin{itemize}
\item[(1)] Let $G = T = (\C^*)^n$ be an algebraic torus. In this case, $G=B=T=(\C^*)^n$, and therefore a spherical $T$-variety is the same as a toric $T$-variety. 
\item[(2)] Let $X = G/P$ be a partial flag variety, equipped with the usual left action of $G$ by multiplication. By the Bruhat decomposition, $X$ is then a spherical $G$-variety. This is an example of a projective spherical $G$-variety for a non-abelian $G$. 
\item[(3)] Let $G = \SL(2,\C)$ and cnsider the natural action of  $G = \SL(2, \C)$ on $\A^2$ by the usual (left) matrix multiplication. It is straightforward to see that $G$ acts transitively on 
$\A^2 \setminus \{(0,0)\}$. Moreover, the stabilizer of the point $(1, 0)$ is the maximal unipotent subgroup $U$ of upper triangular matrices in $\SL(2,\C)$ with $1$'s on the diagonal.
Thus $\A^2 \setminus \{(0,0)\}$ can be identified with the homogeneous space $G/U$. Let $B$ be the subgroup of upper triangular 
matrices in $\SL(2,\C)$. Then it is not hard to see that the $B$-orbit of the point $(0,1)$ is the dense open subset $\{(x, y) \mid y \neq 0\}$ of $\A^2$. Thus, $\A^2 \setminus \{(0,0)\} \cong G/U$ is a spherical variety. This is an example of a quasi-affine (but not affine) spherical variety. Similarly, it can be verified that $\A^n \setminus \{0\}$ is a spherical variety for the natural action of $G = \SL_n(\C)$. 
(However, for $n>2$, the $\SL_n(\C)$-stabilizer of a point in $\A^n \setminus \{0\}$ is larger than a maximal unipotent subgroup).
\item[(4)] More generally, consider $X = G/U$ where $U$ is a maximal unipotent subgroup of $G$, equipped with the left action of $G$. Again by the Bruhat decomposition, $X$ is a spherical $G$-variety, and it is well-known that $X$ is quasi-affine. It is useful to note that there is a natural fiber bundle $G/U \to G/B$, where $B$ is the Borel subgroup containing $U$, with fiber isomorphic to the torus $B/U \cong T$. 
\item[(5)] Let $X = G$ and consider the ``left-right'' action of $G \times G$ on $X=G$ given by $(g,h) \cdot k = g k h^{-1}$. Clearly this action is transitive, and the stabilizer of the identity $e$ is the diagonal subgroup 
$G_{\diag} = \{(g, g) \mid g \in G\} \subseteq G \times G$. Thus $X = G$ can be identified with the homogeneous space $(G \times G) / G_{\diag}$. 
Again from the Bruhat decomposition it follows that $X=G$ is a $(G \times G)$-spherical variety. Here the Borel subgroup of $G \times G$ is chosen to be $B \times B^-$, where $B$ is a Borel subgroup of $G$ and $B^-$ is its opposite. The $(G \times G)$-equivariant completions 
of $G$ are usually called {\it group compactifications}.  
\end{itemize}
\end{Ex}

For the rest of the paper we restrict attention to the setting of spherical homogeneous spaces, i.e., $X=G/H$. 
In the setting of spherical homogeneous spaces $X=G/H$, it turns out there is a convenient way to view the sublattice $\Lambda_X = \Lambda_{G/H}$, as follows. 
Choose $B$ a Borel subgroup so that the $B$-orbit of $eH \in G/H$ is dense in $G/H$; such a $B$ exists since $G/H$ is spherical. Now let $S$ be a maximal torus in the intersection $B \cap H$ and choose $T$ to be a maximal torus in $B$ containing $S$. Define $T_{G/H} := T / S$. It is known that the character lattice of $T_{G/H}$ can be identified with $\Lambda_{G/H}$. 

We next briefly review the theory of valuations on $\C(G/H)$. 
In this manuscript, by a valuation $\nu$ on $\C(G/H)$ we will 
mean a discrete $\Q$-valued valuation on $\C(G/H)/\C$, i.e., we assume
\begin{enumerate} 
\item $\nu: \C(G/H)^* \to \Q$ and $\nu(0) = \infty$, 
\item $\nu(\C(G/H)^*) \cong \Z$ or $\{0\}$, 
\item  $\nu(\C^*) = 0$, 
\item $\nu(fg) = \nu(f) + \nu(g)$, 
\item $\nu(f+g) \geq \min \{ \nu(f), \nu(g) \}$. 
\end{enumerate} 
A valuation $v: \C(G/H) \setminus \{0\} \to \Q$ is \textbf{$G$-invariant} if for any $g \in G$ and $f \in \C(G/H)$ we have $v(f) = v(g \cdot f)$. We define 
\begin{equation}\label{eq: def VGH}
\V_{G/H} := \{ \textup{ $G$-invariant valuations on } \C(G/H) \}. 
\end{equation}

From the Luna-Vust theory, it is known 
that any $G$-invariant valuation in the sense explained above is a geometric valuation, or more precisely, is induced by (up to multiplying by a constant in $\Q$) a divisor (i.e. is the order of vanishing along a divisor).

\begin{Ex}\label{example: torus invariant valuation} 
Let $G=T=(\C^*)^n$ and $H=\{1\}$. Then $G/H = (\C^*)^n$. For any vector $w \in \Q^n$ we can construct a $G=T$-invariant valuation $v_w$ as follows. For $f  = \sum_\alpha c_\alpha x^\alpha \in \C[G/H] \cong \C[T] \cong \C[x_1^{\pm}, \ldots, x_n^{\pm}]$ we define 
\[
v_w(f) := \min \{ w \cdot \alpha \, \mid \, c_\alpha \neq 0 \}.
\]
Then $v_w$ extends to the field of rational functions $\C(G/H)$ and this is easily checked to be $T$-invariant. 
\end{Ex}

Let $v$ be a 
$G$-invariant valuation. By the definition of valuations, $\nu$ gives rise to a homomorphism $\nu: \C(G/H)^{(B)} \to \Q$, where the group operation on $\C(G/H)^{(B)}$ is multiplication of functions. Moreover, by assumption, the valuation evaluates trivially on constant functions, so in fact we obtain a group homomorphism $\C(G/H)^{(B)}/\C^* \to \Q$. The identification $\C(G/H)^{(B)}/\C^* \cong \Lambda_{G/H}$ above then allows us to view this as a linear map $\Lambda_{G/H} \to \Q$. We denote this linear map by $\rho(v) \in \Hom(\Lambda_{G/H}, \Q)$. We introduce the notation $\mathcal{Q}_{G/H} := \Hom(\Lambda_{G/H}, \Q)$ for this $\Q$-vector space. With this notation, the correspondence $v \mapsto \rho(v) \in \Hom(\Lambda_{G/H},\Q)$ is a mapping 
\begin{equation}\label{eq: def rho}
\rho: \V_{G/H} \to \mathcal{Q}_{G/H} := \Hom(\Lambda_{G/H},\Q).
\end{equation} 
The following is well-known \cite[7.4 Prop]{Luna-Vust}. 
\begin{Th}\label{theorem: Luna-Vust} 
 With the assumptions as above, the map $\rho: \V_{G/H} \to \mathcal{Q}_{G/H}$ is injective, i.e., a $G$-invariant valuation is uniquely determined by its restriction to the $B$-eigenfunctions.
 \end{Th} 
 
 Due to the above theorem, we may henceforth identify $\V_{G/H}$ with its image $\rho(\V_{G/H}) \subseteq \mathcal{Q}_{G/H}$. 

Note there is a natural pairing between $G$-invariant valuations and the lattice $\Lambda_{G/H}$ given by 
$\langle v, \lambda \rangle := \rho(v)(\lambda)$.

\begin{Ex} \label{rem-val-weight-vec}
Continuing Example~\ref{example: torus invariant valuation}, the lattice $\Lambda_{G/H}$ in the case $G/H=T=(\C^*)^n$ is isomorphic to $\Z^n$. The $B$-semi-invariant functions $\C(G/H)^{(B)} = \C(T)^{(T)}$ are the monomials $x^\alpha = x_1^{\alpha_1} x_2^{\alpha_2} \cdots x_n^{\alpha_n}$, and the monomial $x^\alpha$ corresponds to the $B=T$-weight $\alpha = (\alpha_1,\ldots,\alpha_n) \in \Z^n \cong \Lambda_{G/H}$. 
In this special case, it is straightforward to explicitly compute the natural pairing described above. 
For a $G(=T)$-invariant valuation $v_w$ as constructed in Example~\ref{example: torus invariant valuation}, the element $\rho(v_w) \in \Hom(\Lambda_{G/H},\Q)$ sends $\alpha$ to $\nu_w(x^\alpha) = w \cdot \alpha$ (by definition of $v_w$), i.e., the usual dot product of $w \in \Q^n$ with $\alpha \in \Z^n$. 
\end{Ex}

The following result is due to Brion \cite{VUGen} and Knop \cite{WeylMom}.


\begin{Th} \label{th-val-cone-co-simplicial}
The set $\V_{G/H}$ is a co-simplicial cone in the vector space $\mathcal{Q}_{G/H}$. Moreover, it is the fundamental domain for the action of a Weyl 
group of a root system. 
More precisely, there exists a set of simple roots $\beta_1, \ldots, \beta_\ell$ in this root system such that the cone $\V_{G/H}$ is defined by:
$$\V_{G/H} = \{ v \in \mathcal{Q}_{G/H} \mid \langle v, \beta_i \rangle \leq 0 , ~\forall i=1, \ldots, \ell \}$$
where the pairing $\langle \cdot, \cdot \rangle$ is the one described above, and the $\beta_i$ lie in $\Lambda_{G/H}$. 
\end{Th}

The set of simple roots $\{\beta_1, \ldots, \beta_\ell\}$ is called the system of {\it spherical roots} of $G/H$.
This Weyl group of the spherical root system is also called the {\it little Weyl group of $G/H$}. 

Henceforth, we additionally assume that the spherical variety $G/H$ is quasi-affine.
\begin{Rem}\label{remark: not severe}
This assumption is not severe. 
Following \cite{Akhiezer-rk} one can show that, by adding an extra $\C^*$ component if necessary, we can assume that $G/H$ is quasi-affine. More precisely, let $\tilde{G} = G \times \C^*$. Take a character $\chi: H \to \C^*$ and define $\tilde{H} = \{ (h, \chi(h)) \mid h \in H\} \subset \tilde{G}$. The homogeneous space $\tilde{G} / \tilde{H}$ is spherical for the action of $\tilde{G}$. In fact, if $B$ is a Borel such that $BH$ is open in $G$ then $\tilde{B} \tilde{H}$ is open in $\tilde{G}$ where $\tilde{B} = B \times \C^*$. Also, one can show that for a suitable choice of a character $\chi$, the homogeneous space $\tilde{G} / \tilde{H}$ can be equivariantly embedded in some finite dimensional $\tilde{G}$-module \cite[Section 11.2]{Humphreys} and is thus quasi-affine. The projection $\tilde{G} \to G$, $(g, z) \mapsto g$, induces a $G$-equivariant morphism $\tilde{G} / \tilde{H} \to G/H$ which is a fiber bundle with fibers isomorphic to $\C^*$.  
\end{Rem} 

Assuming that $G/H$ is quasi-affine, let $\C[G/H]$ denote the ring of regular functions on $G/H$. 
When $\C[G/H]$ is considered as a $G$-representation in the natural manner, it is known that $\C[G/H]$ can be decomposed into finite-dimensional irreducible $G$-representations (see e.g. \cite[Appendix p.250]{Timashev}) so we have 
\begin{equation}\label{eq: decomp reg functions} 
\C[G/H] \cong \bigoplus_{\lambda} W_\lambda
\end{equation}
where $W_\lambda$ denotes the $\lambda$-isotypic component of $\C[G/H]$. 
When $G/H$ is quasi-affine, sphericity of $G/H$ is equivalent to the ring of regular functions $\mathcal{O}(G/H)$ being a multiplicity-free $G$-module \cite[Theorem 25.1(MF5)]{Timashev}. 
The decomposition above motivates the following definition. 

\begin{Def}\label{def: Lambda plus} 
Assuming that $G/H$ is quasi-affine, 
we define $\Lambda^+_{G/H}$ to be the set of highest weights in $\Lambda^+$ which appear in the decomposition~\eqref{eq: decomp reg functions}, i.e., the set of $\lambda$ for which $W_\lambda \neq 0$ in the RHS of~\eqref{eq: decomp reg functions}. Thus, by definition we have 
$\C[G/H] \cong \oplus_{\lambda \in \Lambda^+_{G/H}} W_\lambda$.  Moreover, since $G/H$ is spherical, each $W_\lambda$ for $\lambda \in \Lambda^+_{G/H}$ is an irreducible $G$-module. 
\end{Def} 

We denote by $\Lambda_{G/H}$ the sublattice of $\Lambda$ generated by $\Lambda^+_{G/H}$. There is an analogue of the usual dominant order for the sublattice $\Lambda_{G/H}$, defined as follows. 
Let $\lambda, \mu \in \Lambda_{G/H}$. We say that $\lambda \geq_{G/H} \mu$ if $\mu - \lambda$ is 
a linear combination of the spherical roots with nonnegative coefficients. 
We call $>_{G/H}$ the {\it spherical dominant order}.
One has the following \cite[Section 5]{Knop-LV}. 

\begin{Th} \label{th-coor-ring-multiplication-spherical}
Let $G/H$ be a quasi-affine spherical homogeneous space. 
Let $\C[G/H] \cong \bigoplus_{\lambda \in \Lambda_{G/H}^+} W_\lambda$ be the ring of regular functions on $G/H$.
Let $f \in W_\lambda$, $g \in W_\mu$. Then the product $fg$ lies in 
$$\bigoplus_{\gamma \geq_{G/H} \lambda + \mu} W_\gamma.$$
\end{Th}

Let $\lambda, \mu \in \Lambda_{G/H}^+$ and let $W_\nu$ appear in the product $W_\lambda W_\mu$, where here $W_\lambda W_\mu$ denotes the span of all products of the form $fg$ for $f \in W_\lambda, g \in W_\mu$. 
For such a triple $\lambda, \mu, \nu \subset \Lambda_{G/H}^+$, we call the weight $\lambda + \mu - \nu$ a \textbf{tail}. 
The \textbf{tail cone} $\Tm_{G/H}$ \textbf{of $G/H$} is defined to be the closure of the cone 
in $\Lambda_{G/H} \otimes \Q$ generated by all the tails. We have
the following (see e.g. \cite[Lemma 5.1]{Knop-LV}). 

\begin{Prop} \label{cor-tail-cone-val-cone}
Let $G/H$ be a quasi-affine spherical homogeneous space. Then the tail cone $\Tm_{G/H}$ is the dual cone to the (negative of the) valuation cone, $-\V_{G/H}$. 
\end{Prop}

\begin{Rem} \label{rem-tail-cone-sph-roots}
From Theorem~\ref{th-val-cone-co-simplicial} and Theorem \ref{th-coor-ring-multiplication-spherical}, we see that the set of spherical roots generates the tail cone $\Tm_{G/H}$ of $G/H$. 
\end{Rem}

\subsection{Spherical tropicalization}\label{subsec: spherical trop}

We start by recalling the notions of a germ of a curve and a formal curve (see for example \cite[Section 24]{Timashev}). 
We let $\mathcal{O} = \C[[t]]$ denote the algebra of formal power series with coefficients in $\C$ and 
$\K = \C((t))$ its field of fractions, i.e. the field of formal Laurent series with finitely many negative exponents. If $f \in \K$ we denote by $\ord_t(f)$ the order of $t$ in the Laurent series $f$. 
Clearly $\ord_t$ is a $\Z$-valued valuation on the field $\K$.

Let $X$ be a variety. 
 A {\it formal curve} $\gamma$ on $X$ is a $\K$-point of $X$. An $\mathcal{O}$-point on $X$ is called a {\it convergent formal curve}. The limit of a convergent formal curve is the point on $X(\C)$ obtained by setting $t=0$ in $\gamma$. 
If we assume $X$ is embedded in an affine space $\A^N$ 
then a formal curve $\gamma$ on $X$ is an $N$-tuple of Laurent series satisfying the defining equations of $X$ in $\A^N$. If $\gamma$ is convergent then
its coordinates are power series and their constant terms are the coordinates of the limit point $\gamma_0 = \lim_{t \to 0} \gamma(t)$. 

\begin{Def}[Valuation associated to a formal curve] \label{def-val-formal-curve}
A formal curve $\gamma$ on $X$ defines a valuation $\hat{v}_\gamma: \C(X)\setminus \{0\} \to \Z \cup \{\infty\}$ as follows. 
\begin{equation} \label{equ-val-formal-curve}
\hat{v}_\gamma(f) = \ord_t(f(\gamma(t))).
\end{equation}
\end{Def}

We recall that the algebraic closure $\overline{\K}$ of the field $\K$ is the field of formal Puiseux series with coefficients in $\C$, the elements of which are formal series 
\[
\gamma(t) = c_1 t^{a_1} + c_2 t^{a_2} + c_3 t^{a_3} + \cdots 
\]
where the $c_i$ are non-zero complex numbers for all $i$, and $a_1 < a_2 < a_3 < \cdots $ are rational numbers that have a common denominator. (Sometimes $\overline{\K}$ is denoted as $\C\{\{t\}\}$ in the literature.) We call a point in $X(\overline{\K})$ a {\it formal Puiseux curve} or simply a {\it Puiseux curve} on $X$. Definition~\ref{def-val-formal-curve} extends naturally to Puiseux curves. That is, a formal Puiseux curve $\gamma$ on $X$ gives a valuation $\hat{v}_\gamma: \C(X) \to \Q \cup \{\infty\}$, defined by the same equation 
\eqref{equ-val-formal-curve}.

Now we restrict attention to the case of spherical varieties and $G$-invariant valuations. Following the setting in the previous section, we assume $G/H$ is a quasi-affine spherical homogeneous space.
The main ingredient in the definition of spherical tropicalization, Definition~\ref{def: strop} below, is the construction of a $G$-invariant valuation from a given arbitrary valuation on ${G/H}$.
The following well-known result is key (see \cite[Lemma 1.4]{Knop-LV}, \cite[Lemma 10 and 11]{Sumihiro}, \cite[3.2 Lemme]{Luna-Vust}).

\begin{Th}[Sumihiro] \label{th-Sumihiro}
Let $G/H$ be a quasi-affine spherical homogeneous space. Let $v: \C({G/H}) \setminus \{0\} \to \Q$ be a valuation. 
\begin{itemize}
\item[(1)] For every $0 \neq f \in \C({G/H})$, there exists a nonempty Zariski-open subset $U_f \subset G$ such that the value $v(g \cdot f)$ is the same 
for all $g \in U_f$. We denote this value by $\bar{v}(f)$, i.e.
$$\bar{v}(f) := v(g \cdot f), \quad \forall g \in U_f.$$
\item[(2)] For $\bar{v}$ defined as above, we have $\bar{v}(f) = \min\{ v(g \cdot f) \mid g \in G \}$.
\item[(3)] $\bar{v}$ is a $G$-invariant valuation on ${G/H}$. 
\end{itemize}
\end{Th}

Recall that a formal curve $\gamma$ on ${G/H}$ gives rise to a valuation $\hat{v}_\gamma$. 
We let $v_\gamma$ denote the $G$-invariant valuation constructed by Theorem~\ref{th-Sumihiro} from $\hat{v}_\gamma$. Concretely, 
$$v_\gamma(f) = \ord_t(f(g \cdot \gamma(t)))$$ for every $f \in \C(G/H) \setminus \{0\}$ and $g \in U_f$ is the 
dense open set in $G$ given in Theorem~\ref{th-Sumihiro}(1). By Theorem~\ref{th-Sumihiro} we know $v_\gamma$ is a $G$-invariant valuation, i.e., $v_\gamma \in \V_{G/H}$. 

\begin{Def} \label{def: strop} 
Let $G/H$ be a quasi-affine spherical homogeneous space. Following \cite{Vogiannou}, we call the map $$\strop: G/H(\overline{\K}) \to \V_{G/H}, \quad \gamma \mapsto v_\gamma,$$ the {\it spherical tropicalization map}. 
\end{Def} 

As in Section~\ref{sec-prelim}, it follows from Luna-Vust theory that the map $\strop$ is surjective. 

\begin{Ex} \label{ex-SL(2)-inv-val}
Let $X = \A^2 \setminus \{(0,0)\}$, equipped 
with the natural action of $G = \SL(2, \C)$ as in Example \ref{ex-sph-var}(3). Thus $X \cong \SL_2(\C)/U$ and 
the algebra of regular functions $\C[X] = \C[\A^2 \setminus \{(0,0)\}$ is the polynomial ring $\C[x, y]$. 
It is not difficult to see that $\Lambda_{G/H}$ coincides with the weight lattice $\Lambda \cong \Z$ of $SL_2(\C)$. Indeed, the function $f(x,y) = y$ on $X =\A^2 \setminus \{(0,0)\}$ is a $B$-eigenfunction in $\C[X]=\C[x,y]$ whose $B$-weight is $1$. 
Let $\gamma = (\gamma_1, \gamma_2)$ be a formal curve in $X = \A^2 \setminus \{0\}$, where 
$\gamma_1(t) = \sum_i a_i t^i$ and $\gamma_2(t) = \sum_i b_i t^i$ are elements of $\mathcal{K} = \C((t))$.  Let $g = \left[ \begin{matrix} g_{11} & g_{12} \\ g_{21} & g_{22} \end{matrix} \right]$. 
We compute that $f(g \cdot \gamma(t)) = g_{21} \gamma_1 + g_{22} \gamma_2$. 
From the construction of the $G$-invariant valuation $v_\gamma$ above, we know $v_\gamma(y) = \ord_t(f(g \cdot \gamma(t)))$ for $g \in U_f$ for some Zariski-open subset. It follows that
\begin{equation} \label{equ-v-gamma-A^2}
v_\gamma(y) = \min(\ord_t(\gamma_1(t)), \ord_t(\gamma_2(t)).
\end{equation}
\end{Ex}

There is another way of understanding the $G$-invariant valuation associated to a formal curve and that is through the generalized Cartan decomposition for spherical varieties.
It goes back to Luna and Vust (\cite{Luna-Vust}). 
A proof of it can also be found in \cite[Theorem 8.2.9]{Gaitsgory-Nadler}.
\begin{Th}[Generalized non-Archimedean Cartan decomposition for spherical varieties over $\K$]  \label{th-non-Arch-Cartan-decomp} 
The $G(\O)$-orbits in ${G/H}(\K)$ are parameterized by $\check{\Lambda}_{G/H} \cap \V_{G/H}$. Here 
$\check{\Lambda}_{G/H} \subset \mathcal{Q}_{G/H}$ denotes the lattice dual to the $B$-weight lattice $\Lambda_{G/H}$, and a cocharacter 
$\lambda \in \check{\Lambda}_{G/H}$ 
corresponds to the $G(\O)$-orbit through the formal curve $\lambda(t) \in T_{G/H}(\K)$.
\end{Th}

{Thus the valuation $v_\gamma$ can be interpreted as the valuation represented by the point of intersection of the $G(\mathcal{O})$-obit of $\gamma$ in 
${G/H}(\K)$ and the image of valuation cone $\V_{G/H}$ (under the exponential map) in ${G/H}(\K)$.}  

\begin{Ex}[non-Archimedean Cartan decomposition] \label{ex-non-Arch-Cartan-decomp}
As in Example \ref{ex-sph-var}(5) consider $G$ with left-right action of $G \times G$. Theorem \ref{th-non-Arch-Cartan-decomp} applied in this case recovers the
a non-Archimedean version of the usual Cartan decomposition (see \cite{Iwahori}). With notation as above, it states that:
$$G(\K) = G(\O) \cdot \check{\Lambda}^+ \cdot G(\O).$$
{Here $\check{\Lambda}$ is the cocharacter lattice and $\check{\Lambda}^+$ is the intersection of $\check{\Lambda}$ with the dual positive Weyl chamber. We regard both as subsets of $T(\K)$.} 

When $G = \GL(n, \C)$ the above non-Archimedean Cartan decomposition gives the well-known Smith normal form of a matrix (over the 
field of formal Laurent series $\K$ which is the field of fractions of the principal ideal domain $\O$, the ring of formal power series).
\end{Ex}

\begin{Ex}[Non-Archimedean Iwasawa decomposition] \label{ex-non-Arch-Iwasawa}
As in Example \ref{ex-sph-var}(4) consider the spherical homogeneous space $G/U$ where $U$ is a maximal unipotent subgroup of $G$. In this case 
Theorem \ref{th-non-Arch-Cartan-decomp} gives a non-Archimdean version of the Iwasawa decompostion (see \cite{Iwahori}). It states that: 
$$G(\K) = G(\O) \cdot \check{\Lambda} \cdot U(\K),$$
where as in the previous example, $\check{\Lambda}$ is the dual lattice to the weight lattice $\Lambda$ and we regard it as a subset of $T(\K)$.
\end{Ex}

\section{Spherical functions, the spherical logarithm, and separating $K$-orbits}\label{sec: spherical functions}

In this section, following Akhiezer \cite{Akhiezer-rk}, we define certain $K$-invariant real-valued functions on a spherical homogeneous space. As in Section~\ref{sec-prelim}, we assume throughout that $G/H$ is a \emph{quasi-affine} spherical homogeneous space. 
Let $K$ be a maximal compact subgroup of $G$.  

Let $V_\lambda$ be a (non-zero) irreducible $G$-representation in $\C[G/H]$. Fix a $K$-invariant Hermitian product $\langle \cdot, \cdot \rangle$ on $V_\lambda \subset \C[G/H]$. Let $\{f_{\lambda, i}\} \subset V_\lambda$ be an orthonormal basis with respect to this $K$-invariant Hermitian product. Define the function $\phi_\lambda \colon G/H \to \R_{> 0}$ by:
\begin{equation}\label{eq: def spherical}
  \phi_\lambda(x) = \sum_i |f_{\lambda, i}(x)|^2 \text{.}
\end{equation}

\begin{Lem}\label{lemma: phi eH not zero} 
$\phi_\lambda(eH) \neq 0$. 
\end{Lem} 

\begin{proof} 
Suppose $\phi_\lambda(eH) = 0$. From the definition it follows that $f_{\lambda, i}(eH)=0$ for all $i$. Since $V_\lambda$ is a $G$-representation, for any fixed $j$, we know $g \cdot f_{\lambda, j}$ is a linear combination of $f_{\lambda,i}$'s, but this implies $(g \cdot f_{\lambda, j})(eH) := f_{\lambda, j}(g^{-1}H) = 0$ also. Since this was true of any $j$, we conclude that $\phi_\lambda(g^{-1}H) = 0$ for all $g \in G$. This implies $f_{\lambda, i} \equiv 0$ as functions on $G/H$, for all $i$. This is a contradiction, since $V_\lambda$ is a non-zero representation.
\end{proof} 

By the above lemma, we may normalize our choice of $\phi_\lambda$ so that it satisfies 
\[
  \phi_\lambda(eH) = 1,  \quad \forall \lambda \in \Lambda^+_{G/H} \text{.}
\]

\begin{Rem}    \label{rem-weight-vec-orth-basis}
It is not hard to see that, with respect to any $K$-invariant Hermitian product $\langle \cdot, \cdot \rangle$ on $V_\lambda$, two $T$-weight vectors corresponding to different $T$-weights must be orthogonal with respect to $\langle \cdot, \cdot \rangle$. Thus we may assume without loss of generality that the orthonormal basis $\{f_{\lambda, i}\}$ consists of $T$-weight vectors. 
\end{Rem}

The following lemma is straightforward. 

\begin{Lem} \label{lem-sph-fct}
  The function $\phi_\lambda$ is $K$-invariant, and independent of the chosen orthonormal basis $\{f_{\lambda, i}\}$.
\end{Lem}

We call the function $\phi_\lambda$ the \textbf{spherical function associated to the highest weight $\lambda \in \Lambda^+_{G/H}$}. 
Before stating the next result we need some preliminaries. First we consider the algebra of $K$-invariants $(\C[G/H] \otimes \overline{\C[G/H]})^K$.  
By multiplication of functions, we have a natural map $\C[G/H] \otimes \overline{\C[G/H]} \to \C[G/H] \cdot \overline{\C[G/H]}$ where the target denotes the $\C$-span of all $\C$-valued functions on $G/H$ of the form $f \cdot \bar{g}$ where $f, g \in \C[G/H]$. It is possible to see, using the fact that holomorphic functions are determined by their values on any non-empty open neighborhood, and the independence of holomorphic and anti-holomorphic variables $z_i$ and $\bar{z}_i$ in some local coordinates, that this map is an isomorphism.
Therefore, in what follows, we will slightly abuse notation and sometimes write $\C[G/H] \otimes \overline{\C[G/H]}$ and sometimes write its image $\C[G/H] \cdot \overline{\C[G/H]}$. In both cases, the $G$-action is given by a diagonal $G$-action on each factor, and the identification is $G$-equivariant. 

The following is a result of Akhiezer \cite[Lemma 3]{Akhiezer-rk}. 

\begin{Prop}{\cite[Lemma 3]{Akhiezer-rk}}
  \label{prop-sph-functions}
  The set of functions $\left\{ \phi_\lambda \mid \lambda \in \Lambda^+_{G/H} \right\}$ is a basis for $(\C[G/H] \cdot \overline{\C[G/H]})^K$ as a vector space over $\C$. In particular, 
\begin{equation} \label{equ-G/H-K}  
(\C[G/H] \cdot \overline{\C[G/H]})^K = \bigoplus_{\lambda \in \Lambda^+_{G/H}} \C\,\phi_\lambda.
\end{equation}
\end{Prop}

\begin{proof}
  The inclusion ``$\supset$'' follows by Lemma \ref{lem-sph-fct}. For the reverse inclusion, 
  first observe that 
\begin{equation*}   
    \rleft( \C[G/H] \otimes \overline{\C[G/H]} \rright)^K = \bigoplus_{\lambda_1,\lambda_2 \in \Lambda_{G/H}^+} \rleft( V_{\lambda_1} \otimes \overline{V}_{\lambda_2} \rright)^K \cong \bigoplus_{\lambda_1, \lambda_2 \in \Lambda_{G/H}^+} \Hom_K(V_{\lambda_2}, V_{\lambda_1})
\end{equation*}
where $\Hom_K(V_{\lambda_2},V_{\lambda_1})$ denotes the set of $K$-equivariant module homomorphisms and where we have used the fact that, for any rational $K$-module $V$, the complex conjugate $K$-module $\overline{V}$ is isomorphic to the dual $K$-module $V^*$ (the isomorphism depends on the choice of a $K$-invariant Hermitian product on $V$).
  By Schur's lemma, $\Hom_K(V_{\lambda_2},V_{\lambda_1})\cong \C$ if $\lambda_1 = \lambda_2$ and $\Hom_K(V_{\lambda_1}, V_{\lambda_2}) = \{ 0 \}$ otherwise. Thus we conclude that 
  $(\C[G/H] \otimes \overline{\C[G/H]})^K \cong \bigoplus_{\lambda \in \Lambda^+_{G/H}} \Hom_K(V_\lambda, V_\lambda)$. Since $\Hom_K(V_\lambda, V_\lambda) \cong \C$ for each $\lambda \in \Lambda^+_{G/H}$ we 
  see that the claim will follow if we can show that $\phi_\lambda$ lies in the image $V_\lambda \otimes \overline{V_\lambda}$ of $V_\lambda \otimes \overline{V_\lambda}$ under the multiplication map and that it is non-zero. But this follows from the definition~\eqref{eq: def spherical} of $\phi_\lambda$ and Lemma~\ref{lemma: phi eH not zero}. 
\end{proof}

Multiplication of spherical functions can be understood in terms of the tail cone, as follows. 
\begin{Lem}    \label{lem-multiplication and tails}
  Let $\lambda, \mu \in \Lambda_{G/H}^+$. Then
  \[
    \phi_\lambda \cdot \phi_\mu = \sum_{\lambda + \mu - \gamma \in \Tm_{G/H} \cap \Lambda} c_\gamma \phi_\gamma \text{,}
  \]
  for coefficients $c_\gamma \in \R$. Moreover, we have $c_{\lambda+\mu} \neq 0$.
\end{Lem}

\begin{proof} 
We follow the proof of Proposition \ref{prop-sph-functions}. We have:
\begin{equation}
\begin{split} 
\phi_\lambda \phi_\mu \in (V_{\lambda}\overline{V}_{\lambda})^K (V_{\mu}\overline{V}_{\mu})^K &\subset (V_{\lambda}V_{\mu}\overline{V}_{\lambda}\overline{V}_{\mu})^K, \\
&\subset \bigoplus_{\lambda + \mu - \gamma, \lambda + \mu-\gamma' \textup{ are tails } }(V_{\gamma} \overline{V}_{\gamma'})^K,  \quad \textup{ by Theorem~\ref{th-coor-ring-multiplication-spherical}}   \\
&\subset \bigoplus_{\lambda + \mu - \gamma \textup{ is a tail }} (V_{\gamma} \overline{V}_{\gamma})^K   \\
\end{split} 
\end{equation} 
where the last inclusion is by Schur's lemma and the fact that the multiplication map $\C[G/H] \otimes \overline{\C[G/H]} \to \C[G/H] \cdot \overline{\C[G/H]}$ is an isomorphism (so we may identify $V_\gamma \overline{V_{\gamma'}}$ with $V_\gamma \otimes \overline{V_{\gamma'}}$). This proves the first claim. 
To prove that $c_{\lambda+\mu}\neq 0$, we consider the space 
\begin{equation}\label{eq: 2.5} 
\bigoplus_{\lambda + \mu - \gamma \textup{ is a tail}} V_\gamma \otimes \overline{V_\gamma}
\end{equation}
(so this is the space appearing above, before taking $K$-invariants) as a $T_K \times T_K$-representation, for $T_K$ a maximal torus of $K$. From Theorem~\ref{th-coor-ring-multiplication-spherical} and the definition of a tail, we know that $\gamma \leq \lambda+\mu$ in the usual dominance order for any $\gamma$ that appears in the direct sum~\eqref{eq: 2.5}. Since for $\gamma$ any weight, $V_\gamma$ denotes the irreducible $G$-representation with \emph{highest} weight $\gamma$, it follows that the highest $T_K \times T_K$-weight that can appear in~\eqref{eq: 2.5} is $(\lambda+\mu, -(\lambda+\mu))$ (where the partial ordering is $(\lambda, -\mu) > (\lambda', -\mu')$ if $\lambda>\lambda', \mu> \mu'$ in the usual dominance order). Moreover, from the above it follows that the only summand in~\eqref{eq: 2.5} which contains $(\lambda+\mu, -(\lambda+\mu))$ as a $T_K \times T_K$-weight is $V_{\lambda+\mu} \otimes \overline{V_{\lambda+\mu}}$, i.e. the term corresponding to $\gamma=\lambda+\mu$. 

Recall that the spherical function $\phi_{\lambda+\mu}$ is a sum $\sum f_{\lambda+\mu, i} \overline{f}_{\lambda+\mu, i}$ (identified with $\sum f_{\lambda+\mu, i} \otimes \overline{f}_{\lambda+\mu, i}$), where $\{f_{\lambda+\mu, i}\}$ is a basis of $T$-weight vectors in $V_{\lambda+\mu}$. In particular, $\phi_{\lambda+\mu}$ has a non-zero component in the one-dimensional $T_K \times T_K$-weight space of weight $(\lambda+\mu, -(\lambda+\mu))$, i.e. the highest-weight space of $V_{\lambda+\mu} \otimes \overline{V_{\lambda+\mu}}$. It follows from the above discussion that in order to determine whether $c_{\lambda+\mu} \neq 0$, it suffices to show that $\phi_\lambda \phi_\mu$ has a non-zero component in the $T_K \times T_K$-weight space in~\eqref{eq: 2.5} of weight $(\lambda+\mu, -(\lambda+\mu))$. For the following we temporarily denote by $f_\lambda$ (respectively $f_\mu$) the highest weight vector of $T_K$-weight $\lambda$ (respectively $\mu$) in $V_\lambda$ (respectively $V_\mu$). By definition, $\phi_\lambda = f_\lambda \overline{f_\lambda} + \textup{ terms with $T_K$-weight $< (\lambda, -\lambda)$}$ and similarly for $\phi_\mu$. Hence $\phi_\lambda \phi_\mu = f_\lambda \overline{f_\lambda} f_\mu \overline{f_\mu} + \textup{ terms with $T_K \times T_K$-weight $<(\lambda+\mu, -(\lambda+\mu))$}$. Since $f_\lambda \overline{f_\lambda} f_\mu \overline{f_\mu}$ has weight $(\lambda+\mu, -(\lambda+\mu))$ and is non-zero, and all other terms are strictly smaller, there can be no cancellation and we conclude $c_{\lambda+\mu} \neq 0$ as desired. 
\end{proof}

\begin{Prop} \label{prop: algebra generators}
  If $\lambda_1, \ldots, \lambda_s \in \Lambda_{G/H}^+$ generate $\Lambda^+_{G/H}$ as a semigroup, then $\phi_{\lambda_1}, \ldots, \phi_{\lambda_s}$ generate $(\C[G/H] \otimes \overline{\C[G/H]})^K$ as an algebra.
\end{Prop}
\begin{proof}
The proof is a standard ``canceling the leading term argument''. First we consider the case where $G$ is semi-simple. One can construct a well-ordering $>$ on the semigroup $\Lambda^+_{G/H}$ respecting the addition, such that if $\lambda_1 - \lambda_2$ is a non-negative combination of spherical roots then $\lambda_1 > \lambda_2$. Take a vector $\xi$ in the interior of the dual cone to the positive Weyl chamber. Moreover, assume that $\xi$ is irrational, that is, for any nonzero weight $\lambda$, $\langle \xi, \lambda \rangle \neq 0$. For two weights $\lambda$, $\mu$ define $\lambda > \mu$ if and only if $\langle \xi, \lambda \rangle > \langle \xi, \mu \rangle$. It is easy to see that $>$ has the required properties. 

For $f \in (\C[G/H] \otimes \overline{\C[G/H]})^K$ let $f = \sum_\lambda c_\lambda \phi_\lambda$ and define $$v(f) = \min\{\lambda \mid c_\lambda \neq 0\}.$$ 
For convenience we will write $\phi_i$ instead of $\phi_{\lambda_i}$ for $i=1,\ldots, s$. 
Suppose by contradiction that there exists $f \in (\C[G/H] \otimes \overline{\C[G/H]})^K$ that cannot be represented as a polynomial in the $\phi_i$ and $v(f)$ is minimum among all such $K$-invariant functions. Then since $\lambda_1, \ldots, \lambda_s$ generate $\Lambda^+_{G/H}$ we can find $k_i \geq 0$ such that 
$v(f) = v(\prod_i \phi_i^{k_i})$. It follows that there is $0 \neq c \in \C$ such  that $v(f  - c\prod_i \phi_i^{k_i}) < v(f)$. Thus $f  - c\prod_i \phi_i^{k_i}$, and  hence $f$, can be represented as  a polynomial in the $\phi_i$.  The contradiction proves the claim.

Next suppose $G$ is reductive. Replacing $G$ with a finite cover if necessary we can assume that $G = G' \times Z$ where $G'$ is connected semisimple and $Z$ is the connected component of the center of $G$. Then as semigroups $\Lambda^+_G = \Lambda^+_{G'} \oplus \Lambda_Z$, where $\Lambda_Z$ is the character lattice of the torus $Z$. One shows from the definitions that if $\lambda \in \Lambda^+_{G'} \cap \Lambda^+_{G/H}$ and $\gamma \in \Lambda_Z \cap \Lambda^+_{G/H}$ then
\begin{equation}  \label{equ-phi-lambda-gamma}
\phi_{\lambda+\gamma} = \phi_\lambda \phi_\gamma.
\end{equation}
The claim follows from \eqref{equ-phi-lambda-gamma} and the cancelation of the leading terms argument as in the semisimple case.   
\end{proof}

\begin{Rem} \label{rem-phi_k-vs-average-h}
  Let $h_\lambda \in V_\lambda \subset \C[G/H]$ be a nonzero function. Consider the function $\tilde{h}_\lambda$ obtained by averaging $h_\lambda \overline{h_\lambda}$ over $K$, i.e.,
  \[
    \tilde{h}_\lambda(x) = \int_{K} h_\lambda(k \cdot x) \overline{h_\lambda(k \cdot x)} \diff k \text{,}
  \]
  where $\diff k$ is the Haar measure on $K$. Since (the span of) $V_\lambda \overline{V_{\lambda}}$ is a finite-dimensional vector space, the integral $\tilde{h}_\lambda$ is still an element in $\langle V_\lambda \overline{V_{\lambda}} \rangle \subseteq \C[G/H] \otimes \overline{\C[G/H]}$, and is $K$-invariant by construction. From Proposition \ref{prop-sph-functions}, it follows that $\tilde{h}_\lambda$ is a scalar multiple of $\phi_\lambda$.
\end{Rem}

{\begin{Rem}[Highest weight monoid not necessarily free]  \label{rem-Stephanie}
In general, the highest weight monoid $\Lambda^+_{G/H}$ may not be \emph{freely} generated as a monoid. Here is an example which was communicated to us by St\'ephanie Cupit-Foutou and is inspired by \cite[\S 6.3]{Luna}. Let $G_1=G_2=\SL_2$ and for $i=1, 2$ let $T_i$ be a maximal tori of $G_i$ and $N_i$ its normalizer in $G_i$. Now let $G=G_1 \times G_2$ and $H=\ker(c_1 c_2)$ where $c_1$, $c_2$ are the non-trivial characters of $N_1$ and $N_2$ respectively. Then $G/H$ is spherical and affine, and its weight monoid is not free: it is generated by $2\varpi_1$, $2\varpi_2$ and $\varpi_1+\varpi_2$ where $\varpi_i$ denotes the fundamental weight for $G_i$.  
\end{Rem} 
}


\begin{Prop}  \label{prop-phi-lambda-inequalities}
Let $\lambda, \mu, \gamma \in \Lambda^+_{G/H}$ and suppose $\lambda+\mu-\gamma$ is a tail, i.e. $V_\gamma$ appears in $\rleft\langle V_\lambda V_\mu \rright\rangle \subset \C[G/H]$, then there is constant $c > 0$ such that for all $x \in G/H$ we have:
$$\phi_\lambda (x) \phi_\mu(x) \geq c \phi_\gamma(x).$$
\end{Prop}
\begin{proof}
By definition $\phi_\lambda = \sum_i |f_{\lambda, i}|^2$ and $\phi_\mu = \sum_j |f_{\mu, j}|^2$ where $\{f_{\lambda, i}\}$, $\{f_{\mu, j} \}$ are orthonormal bases for $V_\lambda$, $V_\mu$ respectively. Let $\{ f_{\gamma,k} \}$ be an orthonormal basis for $V_\gamma$. As $V_\gamma \subseteq \langle V_\lambda V_\mu \rangle$, we can write $f_{\gamma,k} = \sum_{i,j} a_{ijk} f_{\lambda,i}f_{\mu,j}$ for some $a_{ijk} \in \C$. For $x \in G/H$, it follows that
  \begin{multline*}
    \sqrt{\sum_k |f_{\gamma,k}(x)|^2}  = \sqrt{\sum_k \Big| \sum_{i,j} a_{ijk} f_{\lambda,i}(x) f_{\mu,j}(x) \Big|^2} \le \sum_{i,j} \sqrt{\sum_k |a_{ijk} f_{\lambda,i}(x) f_{\mu,j}(x)|^2}\\
     = \sum_{i,j} |f_{\lambda,i}(x) f_{\mu,j}(x)| \sqrt{\sum_k |a_{ijk}|^2} \le a \sum_{i,j} |f_{\lambda,i}(x) f_{\mu,j}(x)|\\
     \le a\sqrt{\sum_{i,j}|f_{\lambda,i}(x)|^2} \sqrt{\sum_{i,j} |f_{\mu,j}(x)|^2} = a\sqrt{\dim(V_\lambda)\dim(V_\mu)} \sqrt{\sum_i |f_{\lambda,i}(x)|^2} \sqrt{\sum_j |f_{\mu,j}(x)|^2}\\
  \end{multline*}
  where $a \coloneqq \max\{ \sqrt{\sum_k |a_{ijk}|^2} \}$, the first inequality is the triangle inequality and the third inequality is the Cauchy-Schwarz inequality. Squaring both sides we obtain $\phi_\gamma(x) \leq (const) \cdot \phi_\lambda(x) \phi_\mu(x)$ where the constant is $>0$, so by dividing we obtain the claim. 
\end{proof}

The following examples show that spherical functions $\phi_\lambda$ need not be multiplicative. More precisely, 
for $\lambda, \mu \in \Lambda^+_{G/H}$, $\phi_{\lambda+\mu}$ is not necessarily a scalar multiple of $\phi_\lambda \phi_\mu$.

\begin{Ex}  \label{ex-sph-functions-not-multiplicative}

(1) This example is due to Dmitri Timashev.
Let $n \geq 3$ and let $G=\SL_n(\C)$ and $H=\SL_{n-1}(\C)$. Then  it can be shown that $G/H$ is a spherical variety and can be realized as the affine variety in $\C^n \times (\C^n)^*$ defined by the equation $\langle v,w \rangle =1$.  (All facts claimed here without proof are shown in Section~\ref{subsec:sln-sln-1}.) 
It turns out that $\Lambda^+_{G/H}$ is generated by $\chi_1 = \omega_{1}, \chi_2 = \omega_{n-1}$, corresponding to the standard representation $\C^n$ of $\SL_n(\C)$ and its dual. Let $\phi_1 := \phi_{\chi_1}$ and $\phi_2 := \phi_{\chi_2}$ be the corresponding spherical functions. We wish to compare the product $\phi_1 \phi_2$ with the spherical function $\phi_{\chi_1+\chi_2} = \phi_{\omega_1+\omega_{n-1}}$ corresponding to $V_{\omega_+\omega_{n-1}}$.

Let $\{e_1, \ldots, e_n\}$ and $\{e_1^*, \ldots, e_n^*\}$ denote the standard basis of $\C^n$ and its dual basis of $(\C^n)^*$ respectively. The action of $SL_n(\C)$ on the coordinate ring 
\[
\C[x_1,\ldots, x_n, y_1, \ldots, y_n] / \langle \sum_{i=1}^n x_i y_i = 1 \rangle 
\]
is the standard one, viewing $x_1, \ldots, x_n$ (respectively $y_1, \ldots, y_n$) as the coordinates on $\C^n$ (respectively $(\C^n)^*$) corresponding to $\{e_1,\ldots, e_n\}$ (respectively $\{e_1^*, \ldots, e_n^*\}$). 
It is not difficult to see that the span of the $\{x_1,\ldots,x_n\}$ (respectively $\{y_1, \ldots, y_n\}$) yields a representation of $SL_n(\C)$ isomorphic to $(\C^n)^*$, the dual of the standard representation (respectively $\C^n$, the standard representation). The corresponding spherical functions are therefore 
$\phi_1(v,w)=\sum_{i=1}^n |x_i|^2$, where $x_i$ are the coordinates of $v$, and $\phi_2(v,w)=\sum_{i=1}^n |y_i|^2$, where $y_i$ are the coordinates of $w$. Hence their product is $\sum_{i,j}|x_iy_j|^2$. 

Next we wish to compute the spherical function corresponding to $\chi_1+\chi_2$. Since the irreducible $\SL_n(\C)$-representation of highest weight $\omega_1+\omega_{n-1}$ is the adjoint representation $\mathfrak{sl}_n(\C)$, we can proceed as follows. Suppose there exists a $G$-equivariant embedding $\Psi: G/H \to V_{\omega_1+\omega_{n-1}}^*$. This induces a $G$-equivariant map backwards $(V_{\omega_1+\omega_{n-1}}^*)^* \cong V_{\omega_1+\omega_{n-1}} \to \C[G/H] \cong \oplus_{\lambda \in \Lambda_{G/H}^+} V_\lambda$. (Note that $V_{\omega_1 + \omega_{n-1}} = V_{-w_0(\omega_1+\omega_{n-1})} \cong V^*_{\omega_1 + \omega_{n-1}}$ so in our example we may identify $V_{\omega_1+\omega_{n-1}}^*$ with $V_{\omega_1+\omega_{n-1}}$.) Provided that the embedding is nontrivial, by Schur's Lemma this map must take $V_{\omega_1+\omega_{n-1}}$ to the (isomorphic copy of) $V_{\omega_1+\omega_{n-1}}$ appearing in the decomposition of $\C[G/H]$ into irreducible $G$-modules. The spherical function $\phi_\lambda$ is then the pullback via the embedding $\Psi$ of the sum of the norm-squares of the coordinates on $V_{\omega_1+\omega_{n-1}}$ (with respect to some $SU(n)$-invariant inner product on $\mathfrak{sl}_n(\C)$). 
The Hermitian inner product $X,Y \mapsto \mathrm{trace}(XY^*)$ on $\mathfrak{sl}_n(\C)$ is $U(n)$-invariant with respect to the adjoint action of $\SL_n(\C)$. Next, notice that the map $\Psi: (v,w) \mapsto v \otimes w - \frac{1}{n} \mathbf{Id}_{n \times n}$ gives a $G$-equivariant morphism from $\SL_n(\C)/\SL_{n-1}(\C)$ to $\mathfrak{sl}_n(\C)$, where we view $\C^n \otimes (\C^n)^*$ as $\mathfrak{gl}_n(\C)$. Thus we obtain 
\begin{equation}\label{eq: pullback} 
\begin{split} 
\phi_{\omega_1+\omega_n}(v,w) &= \textup{ pullback of $X \mapsto \mathrm{trace}(XX^*)$ by $\Psi$} \\
& = \mathrm{trace}((v\otimes w - \frac{1}{n} \mathbf{Id}_{n \times n})^*(v\otimes w - \frac{1}{n} \mathbf{Id}_{n \times n})) \\
& = \| v \|^2 \|w\|^2 - \frac{1}{n} \\
& = (\sum_{i,j} \| x_i \|^2 \|y_j\|^2) - \frac{1}{n} 
\end{split} 
\end{equation} 
which is clearly not proportional to $\phi_1 \phi_2$.  

(2) (Group case) Consider $G$ as a $(G \times G)$-spherical homogeneous space. Let $B$ be a Borel subgroup of $G$ and take $B \times B^-$ as  Borel subgroup of $G \times G$ where $B^-$ denotes the opposite Borel subgroup to $B$. Then the highest weight semigroup of $G$ is $\{(\lambda, -\lambda) \mid \lambda \in \Lambda^+\} \subset \Lambda^+ \times \Lambda^-$. For $\lambda \in \Lambda^+$ let us write $\phi_\lambda$ in place of the spherical function $\phi_{(\lambda, -\lambda)}$. By reasoning similar to that in example (1) above, we have 
$$\phi_\lambda(g) = \text{tr}(\pi_\lambda(g)\pi_\lambda(g)^*),$$ 
where $\pi_\lambda: G \to \text{End}(V_\lambda)$ is the irreducible representation of $G$ with highest weight $\lambda$ and $*$ stands for the Hermitian adjunction with respect to some $K$-invariant Hermitian product on the representation space $V_\lambda$. It follows that if $g$ lies in the real locus of a split maximal torus $T' \subset G$, then $\phi_\lambda(g) = \text{tr}(\pi_\lambda(g) \pi_\lambda(g)) = \text{tr}(\pi_\lambda(g^2)) = \chi_{\lambda}(g^2)$ where $\chi_\lambda$ denotes the character of the representation $\pi_\lambda$. But one knows that the product of characters is the character of tensor product. This shows that the multiplicativity does not hold for the $\phi_\lambda$.
\end{Ex}

\begin{Prop}\label{prop: separate K orbits} 
Let $G/H$ be a quasi-affine spherical homogeneous space. Then the spherical functions $\phi_\lambda$, $\lambda \in \Lambda_{G/H}^+$, separate $K$-orbits in $G/H$.  More precisely, for any two distinct (hence disjoint) $K$-orbits $O_1$ and $O_2$ in $G/H$, then there exists some $\lambda \in \Lambda^+_{G/H}$ such that $\phi_\lambda(O_1) \neq \phi_\lambda(O_2)$. 
\end{Prop}

\begin{proof}
Since $G/H$ is quasi-affine, there exists a $G$-equivariant embedding $G/H \hookrightarrow \C^m = \R^{2m}$ for some $m$, where $\C^m$ is equipped with a linear $G$-action (see \cite[Section 1.2]{Popov-Vinberg}). Let $O_1, O_2 \subset G/H$ be two distinct $K$-orbits. Since $K$ is compact, both $O_1, O_2$ are compact. As disjoint compact subsets of an affine space, by Urysohn's lemma, there exists a continuous, real-valued function $f$ on $\R^{2m}$ such that $f(O_1)=0$ and $f(O_2)=1$. By the Stone approximation theorem (see e.g. \cite[Chapter 6]{Cheney}), there exists a real polynomial $p$ on $\R^{2m}$ that approximates $f$ arbitrarily well on any compact subset.
Since $K$ acts linearly on the ambient $\C^m$, we can average this polynomial by integrating over $K$ and obtain a $K$-invariant real polynomial $p$ that separates $O_1$ and $O_2$. The restriction of $p$ to $G/H$ belongs to $(\C[G/H] \otimes \overline{\C[G/H]})^K = \bigoplus_{\lambda \in \Lambda^+_{G/H}} \C\,\phi_\lambda$ and hence we conclude that there is some $\phi_\lambda$ that separates $O_1$ and $O_2$ as required. 
\end{proof}

Let $\Gamma = \{\lambda_1, \ldots, \lambda_s\} \subset \Lambda^+_{G/H}$ be a set of semigroup generators of $\Lambda^+_{G/H}$.  
We have seen in Proposition~\ref{prop: algebra generators} that the collection of associated spherical functions $\phi_{\lambda_1}, \phi_{\lambda_2}, \cdots, \phi_{\lambda_s}$ are algebra generators of the subalgebra $(\C[G/H] \otimes \overline{\C[G/H]})^K$. Thus, it follows from Proposition~\ref{prop: separate K orbits} that they also separate $K$-orbits.  Thus, this set of spherical functions plays an important role in understanding the geometry of $K$-orbits of $G/H$, and motivates the following definition of the \textbf{spherical logarithm map}. 

\begin{Def}   \label{def-Phi}
Let $\Gamma = \{ \lambda_1, \ldots, \lambda_s\} \subseteq \Lambda^+_{G/H}$ be a set of semigroup generators of $\Lambda^+_{G/H}$ and let $\phi_{\lambda_1}, \cdots, \phi_{\lambda_2}$ denote the corresponding spherical functions. We define 
$$\sLog_{A, t}(x) := (\log_t \phi_{\lambda_1}(x), \ldots, \log_t \phi_{\lambda_s}(x)).$$
We call this the \textbf{spherical logarithm map associated to $\Gamma = \{\lambda_1, \ldots, \lambda_s\}$.} 
\end{Def}

{\begin{Rem}\label{rem-slog-horosph}
We already noted in the introduction that when $G/H$ is horospherical, we can define a spherical logarithm map canonically, independent of a choice of $\Gamma$.  In fact, we can say more: in the horospherical case, we also do not need to assume that $G/H$ is quasi-affine. 
Here we briefly sketch how to define the spherical logarithm map without quasi-affineness. 

First consider an $r$-dimensional torus $S \cong (\C^*)^r$ with character lattice $\Lambda_S$ and maximal compact torus $K_S \cong (S^1)^r \subset S$. Then there is a canonical logarithm map $\Log: S \to \Hom(\Lambda_S, \R) \cong \R^r$ defined as follows: for $x \in S$, let $\Log(x)$ be the homomorphism of $\Lambda_S$ that sends any character $\lambda \in \Lambda_S$ to $\log(|\lambda(x)|)$, where $\log$ denotes the natural logarithm. For $t>0$, we then define $\Log_t(x) := (1/\log t)\Log(x)$. It is clear that $\Log_t$ is $K_S$-invariant for any $t>0$.

We now return to the spherical setting. Suppose $G/H$ is horospherical. In this case it is known that $P=N_G(H)$ is a parabolic subgroup. Moreover, for a Levi decomposition $P=P_u L$ of $P$, we have $H = P_u L_0$ for some subgroup $L_0$ where $L' \subset L_0 \subset L$ and $L'$ denotes the commutator subgroup of $L$ \cite[Lemma 7.4]{Timashev}. It follows that the quotient group $S :=P/H$ is abelian and hence a torus. For $\pi: G/H \to G/P$ the natural projection, we then see that all the fibers of $\pi$ are tori isomorphic to $S$, and in particular, 
$\pi^{-1}(eP) = S$. Now let $K$ be a maximal compact subgroup of $G$ such that $KP = G$. Then $K$ acts transitively (from the left) on $G/P$. Since the projection $\pi: G/H \to G/P$ is $K$-equivariant and $K$ acts transitively on the base, it follows that every $K$-orbit in $G/H$ intersects $S=\pi^{-1}(eP)$. This yields a map $\varphi$ from the space of $K$-orbits $K \backslash G/H$ to $S/K_S$ (it turns out the intersection of a $K$-orbit with $S$ is a $K_S$-orbit). We can now define the spherical logarithm map $\sLog_t: G/H \to \Hom(\Lambda_S, \R)$ by composing $\varphi$ with the $\Log_t$ defined in the previous paragraph.  It can be verified that when $G/H$ is quasi-affine this $\sLog_t$ map coincides with the $\sLog_t$ map defined using spherical functions (cf. the discussion after Theorem \ref{th-intro:separate K-orbits and algebra generators}).   
\end{Rem}
}

\section{Spherical tropicalization, spherical amoebae and the limits of spherical logarithms}
\label{sec: limit}

The purpose of this section is to examine relationships between the spherical tropicalization map, introduced in Section~\ref{subsec: spherical trop}, with the spherical logarithm map of Definition~\ref{def-Phi}.  In order to do so, a slight change in perspective is useful, because the spherical logarithm is inherently a real object (it being a logarithm), i.e. defined over $\R$, whereas the valuation cone is naturally defined over $\Q$. More precisely, 
for the discussion that follows, we first choose an explicit identification of $\Hom(\Lambda_{G/H}, \Q) = \mathcal{Q}_{G/H}$ with $\Q^r$ using a choice of basis. Specifically, let $\{\lambda_1, \ldots, \lambda_r\} \subseteq \Lambda^+_{G/H}$ be the basis of $\Lambda_{G/H} \otimes \R$ appearing in Definition~\ref{def-Phi}. This gives us an identification $\Hom(\Lambda_{G/H}, \Q) \cong \Q^r$. Moreover, we additionally extend our coefficents from $\Q$ to $\R$ for the remainder of this manuscript, so we consider $\mathcal{Q}_{G/H} \otimes \R \cong \Hom(\Lambda_{G/H}, \R)$ and the choice of basis above naturally also yields an identification $\mathcal{Q}_{G/H} \otimes \R \cong \R^r$, which we fix throughout. The embedding $\rho: \V_{G/H} \hookrightarrow \mathcal{Q}_{G/H}$ in~\eqref{eq: def rho} then allows us to think of the valuation cone $\V_{G/H}$ as a subset of $\mathcal{Q}_{G/H} \otimes \R \cong \R^r$. We let $\overline{\V}_{G/H}$ denote the closure of $\V_{G/H}$ in $\R^r$ with respect to the usual Euclidean topology on $\R^r$. It is a co-simplicial cone in $\R^r$ with the same defining equations as given in Theorem~\ref{th-val-cone-co-simplicial}. (We slightly abuse terminology and call $\overline{\V}_{G/H}$ also the valuation cone.) 
The reader may note that the spherical logarithm map of Definition~\ref{def-Phi} uses a choice of $\Gamma = \{\lambda_1, \ldots, \lambda_s\}]$ where it is possible that $s>r$. In such a case, the image of $\sLog_{\Gamma,t}$ lies in $\R^s$ and not in $\R^r$, whereas the valuation cone lies in $\mathcal{Q}_{G/H} \otimes \R \cong \R^r$, as just described above. In this situation we will use the 
linear embedding of the valuation cone $\V_{G/H}$ into $\R^s$ specified by $\Gamma$, given by:
$$v \mapsto (\langle v, \lambda_1\rangle, \ldots, \langle v, \lambda_s\rangle) \in \R^s, \quad v \in \V_{G/H}.$$ 
In this way we can embed $\V_{G/H}$ in $\R^s$ and identify it with its image in $\R^s$, thus allowing us to compare $\V_{G/H}$ with the image of the spherical logarithm. 

Our first result shows that points in the valuation cone can be realized as limits of points in the image of the spherical logarithm map.
As in the previous section, we assume here that $G/H$ is a quasi-affine spherical homogeneous space. 
Let $B$ be the Borel subgroup of $G$ such that the $B$-orbit of $eH$ is open in $G/H$. Note that the multiplication map $K \times B \to G$ (similar to the Iwasawa decomposition) is a surjection, with kernel isomorphic to $K \cap B$.  

\begin{Th}   \label{th-phi-approach-inv-valuation}
Let $G/H$ be a quasi-affine spherical homogeneous space. 
Let $\gamma \in G/H(\C[[t]])$ be a convergent formal Laurent curve in $G/H$ which is convergent for sufficiently small $t \neq 0$. Let $v_\gamma \in \V_{G/H}$ be the $G$-invariant valuation associated to the formal curve $\gamma$. Let $\lambda \in \Lambda^+_{G/H}$ and let $\phi_\lambda: G/H \to \R_{>0}$ the corresponding spherical function. Then for any highest weight vector $f_\lambda \in V_\lambda \subset \C[G/H]$ we have:
\begin{equation}   \label{equ-v-lambda-f-lambda}
\lim_{t \to 0} \log_t(\phi_\lambda(\gamma(t))) = 2 v_\gamma (f_\lambda).
\end{equation}
\end{Th}

\begin{proof}
Let $\lambda \in \Lambda^+_{G/H}$ be fixed throughout. By Theorem~\ref{th-Sumihiro} there exists  
an open set $U \subset G$ such that for $g \in U$ we have 
$v_\gamma(f_\lambda) = \hat{v}_\gamma(g \cdot f_\lambda)$ where $\hat{v}_\gamma$ computes
is the order in $t$ of $g \cdot f_\lambda$ after restricting to the curve $\gamma$. 
We claim that there is an open dense subset $U' \subset K$ such that for $k \in U'$ we have 
\begin{equation}   \label{equ-v-gamma-k}
v_\gamma(f_\lambda) = \hat{v}_\gamma(k \cdot f_\lambda).
\end{equation}
To see this, first note that for $g \in U$ if we write $g = kb$ with $b \in B$ and $k \in K$ then since $f_\lambda$ is a $B$-weight vector we have $v_\gamma(f_\lambda) = \hat{v}_\gamma(k \cdot f_\lambda)$. Now let $m: K \times B \to G$ be the multiplication map and $pr_1: B \times K \to K$ the projection on the first factor. It suffices to take $U' = pr_1(m^{-1}(U))$. Clearly, $m^{-1}(U)$ is open. Note that the multiplication map $m: K \times B \to G$ is a fibration with fibers isomorphic to $B \cap K$. But in a fibration the inverse image of a dense subset is dense. This shows that $m^{-1}(U)$ is dense in $K \times B$ as well. Since the projection map $pr_1$ is open, it follows that $U' = pr_1(m^{-1}(U))$ is open and dense in $K$, which proves the claim. 

Now let us define the $K$-invariant function $\psi_\lambda: G/H \to \R_{>0}$ by
\begin{equation}\label{eq: Psi lambda}
\psi_\lambda(x) = \int_{K} f_\lambda(k^{-1} \cdot x) \overline{f_\lambda(k^{-1} \cdot x)} dk
\end{equation}
 where $dk$ is the Haar measure on $K$. By Remark \ref{rem-phi_k-vs-average-h} we know that $\psi_\lambda$ is a scalar multiple of the spherical function $\phi_\lambda$. 
 
Let $a_\lambda := v_\gamma(f_\lambda)$ denote the value of the invariant valuation $v_\gamma$ on $f_\lambda$.  
Now we claim that the limit $$\lim_{t \to 0} \frac{\psi_\lambda(\gamma(t))}{t^{2a_\lambda}}$$ exists and is nonzero. 
We have just seen in \eqref{equ-v-gamma-k} that for $k \in U'$ we have 
\[
(k \cdot f_\lambda)(\gamma(t)) = c(k) t^{a_\lambda} + \textup{ higher-order terms}
\]
 for some constant $c(k) \in \C, c(k) \neq 0$. 
Hence $|(k \cdot f_\lambda)(\gamma(t))|^2 = |c(k)|^2 t^{2a_\lambda} + \textup{ higher-order terms}$, for $t \in \R$.  
From the definition~\ref{eq: Psi lambda} of $\Psi_\lambda$ we have 
\begin{align*}  
\psi_\lambda(\gamma(t)) &= \int_K |(k \cdot f_\lambda)(\gamma(t)|^2 dk,\\
& = \int_{U'} |(k \cdot f_\lambda)(\gamma(t))|^2 dk, 
\quad \textup{because } U' \subset K \textup{ is open and dense},\\ & = (\int_{U'} |c(k)|^2 dk) t^{2a_\lambda} + \textup{ higher terms},\\
& = c\, t^{2a_\lambda} + \textup{ higher terms}, \label{equ-psi-t-expansion}
\end{align*}
where $c = \int_{U'} |c(k)|^2 dk > 0$, and the term-by-term integration which is implicit in the last equality is justified because, from the compactness of $K$, it follows that for sufficiently small $t \neq 0$, the series $c(k) t^{a_\lambda} + \cdots$ converges uniformly in $k \in K$. Finally, since $\phi_\lambda$ is a constant multiple of $\psi_\lambda$, from the above we have:
\begin{align*}   
\lim_{t \to 0} \log_t(\phi_\lambda(\gamma(t))) &= \lim_{t \to 0} \log_t(\psi_\lambda(\gamma(t))) \\
& = \lim_{t \to 0} \log_t(t^{2a_\lambda}(c + \textup{(terms of positive order in $t$)})) \\
& = 2a_\lambda + \lim_{t \to 0} \frac{\ln(c + \textup{(terms of positive order in $t$)})}{\ln(t)} \\
&= 2a_\lambda \\
& = 2 v_\gamma(f_\lambda) \\
\end{align*}
where the last equality is the definition of $a_\lambda$. 
This finishes the proof.
\end{proof}

{\begin{Rem} \label{rem-KM}
In \cite[Conjecture 7.1]{Kaveh-Manon}, the fourth author and Christopher Manon conjecture an analogue of the Cartan decomposition for a spherical homogeneous space over $\C$. They also suggest a definition of a spherical logarithm map on $G/H$ whenever such a Cartan decomposition exists. In particular, in \cite{Kaveh-Makhnatch}, the singular value decomposition theorem (i.e. the Cartan decomposition for $\GL_n(\C)$) is used to propose a definition of a spherical logarithm map on the group $\GL_n(\C)$, and it is proven that the logarithm of singular values of a curve $\gamma(t)$ on $\GL_n(\C)$ approach its invariant factors. 
This result can be regarded as a special case of Theorem~\ref{th-phi-approach-inv-valuation}. It should be noted that the definition of spherical logarithm proposed in \cite{Kaveh-Manon} is a priori different from that given in this paper -- for instance, in the ``group case'' $\GL_n(\C)$, our definition gives (the logarithms of) symmetric functions in the singular values (cf. Section~\ref{subsec-ex-gp-case}), whereas the definition in \cite{Kaveh-Manon} gives (logarithms of) the singular values. 
\end{Rem}
 }

Theorem~\ref{th-phi-approach-inv-valuation} holds for any weight $\lambda \in \Lambda^+_{G/H}$. 
By restricting attention to $\lambda \in \Gamma = \{\lambda_1, \ldots, \lambda_s\}$ it follows that 
\[
\lim_{t \to 0} \sLog_{\Gamma, t}(\gamma(t)) = 2 \cdot \strop(\gamma(t))
\]
where we think of the RHS as an element of $\V_{G/H} \subseteq \mathcal{Q}_{G/H} \otimes \R \cong \R^r$ as explained above. Thus Theorem~\ref{th-phi-approach-inv-valuation} says that, for certain curves $\gamma \in G/H(\overline{\mathcal{K}})$, the image of $\gamma$ under the spherical tropicalization map can also be realized as the limit of points in the image of the spherical logarithm map. This is reminiscent of the phenomenon seen in the classical case of $G=T$, recounted in the Introduction, where the amoeba $\mathcal{A}_t(Y)$ -- the image under $Log_t$ map of $Y \subseteq (\C^*)^n$ -- approaches the image $\trop(Y)$ of $Y(\overline{\mathcal{K}})$ under the valuation map. Motivated by the classical case and by Theorem~\ref{th-phi-approach-inv-valuation}, we formulate below a question relating the spherical analogues of $Log_t$ and $\trop$.

To give precise statements, we need the notion of \textbf{Kuratowski convergence of subsets of a topological space}, which can be thought of as a notion of continuity for the association $b \mapsto X_b$.

\begin{Def}[Kuratowski convergence (see e.g. \cite{Kuratowski, Jonsson})]  \label{def-Kuratowski}
Suppose $M$ and $B$ are topological spaces. Let $X \subset M \times B$ and let $\pi: X \to B$ be the restriction of the projection $M \times B \to B$ to $X$. For $b \in B$ let $X_b = \pi^{-1}(b)$.  
Let $b_0 \in B$ and $X_0 \subset X$. We say that \textbf{$X_b$ converges to $X_0$ in the sense of Kuratowski
as $b$ converges to $b_0$} if the following conditions hold:
\begin{itemize}
\item[(a)] For every $y \in M \setminus X_0$ there exist neighborhoods $U$ of $y$ and $V_0$ of $b_0$ such that $X_b \cap U = \emptyset$ for all $b \in V_0$. 
\item[(b)] For every $y \in X_0$ and any neighborhood $U$ of $y$ in $M$, there exists a neighborhood $V_0$ of $b_0$ such that $X_b \cap U \neq \emptyset$ for all $b \in V_0$.
\end{itemize}
\end{Def}


Suppose $G/H$ is a quasi-affine spherical variety and $Y \subset G/H$ is a subvariety.  Following the terminology in the classical case, we define the \textbf{spherical amoebae} $\mathcal{A}_t(Y) := \sLog_t(Y)$ to be the images in $\R^r$ of $Y$ under the spherical logarithm maps. We call $\strop(Y)$ the \textbf{spherical tropicalization} of $Y$, considered as a subset of $\R^r$ via the identification $\Lambda_{G/H} \otimes \R$ with $\R^r$ discussed previously. With this terminology in place, we may ask the following. 


\begin{Question}   \label{question: sph-amoeba-approach-sph-trop}
Let $Y \subset G/H$ be a subvariety as above and $\Gamma = \{\lambda_1, \ldots, \lambda_s\} \subset \Lambda^+_{G/H}$ be a set of semigroup generators of $\Lambda^+_{G/H}$. Then, as $t \to 0$, we ask:  
\begin{enumerate} 
\item  Under what conditions do the sets $\sLog_{\Gamma, t}(G/H)$ approach the valuation cone $\overline{\V}_{G/H}$ in the sense of Kuratowski? 
\item Assuming that the images $\sLog_{\Gamma, t}(G/H)$ converge to $\overline{\V}_{G/H}$ in the sense of Kuratowski, under what additional conditions on the subvariety $Y$ do the spherical amoebae $\mathcal{A}_t(Y) = \sLog_t(Y)$ approach $\strop(Y)$ in the sense of Kuratowski? 
\end{enumerate} 
\end{Question}

We note here Theorem~\ref{th-phi-approach-inv-valuation} already gives us some information about Question~\ref{question: sph-amoeba-approach-sph-trop}. The following are straightforward corollaries of Theorem~\ref{th-phi-approach-inv-valuation}. 

\begin{Cor}\label{corollary Kuratowski} 
Let $G/H$ be a quasi-affine spherical homogeneous space. Let $\Gamma = \{\lambda_1, \cdots, \lambda_s\} \subseteq \Lambda^+_{G/H}$ be a set of semigroup generators of $\Lambda^+_{G/H}$.  Let $\V_{G/H}$ denote the valuation cone of $G/H$. 
\begin{enumerate} 
\item  The Kuratowski limit, as $t \to 0$, of the images $\sLog_{\Gamma, t}(G/H)$ contains the valuation cone.
\item When $G/H$ is horospherical this limit is the entire vector space, which in this case coincides with the valuation cone. 
\end{enumerate} 
\end{Cor}

\begin{proof} 
The first claim follows immediately from Theorem~\ref{th-phi-approach-inv-valuation}. The second claim uses the fact that, in the horospherical case, the 
valuation cone $\V_{G/H}$ is the entire vector space $\R^r$. Since by (1) the limit of the images must contain the valuation cone, this means the limit must be all of $\R^r$. 
\end{proof} 

{Although we do not have a complete answer to Question~\ref{question: sph-amoeba-approach-sph-trop}(1), in the remark below we give a sketch of an argument which may be useful. 

\begin{Rem} \label{rem-image-sph-log-conv-val-cone}
It may be possible to answer Question \ref{question: sph-amoeba-approach-sph-trop} (1) positively
using Proposition \ref{prop-phi-lambda-inequalities} as follows. 
Suppose $y$ lies in the Kuratowski limit as $t \to 0$ of the sets $\sLog_{\Gamma, t}(G/H)$. Then by definition there exist sequences $(t_i)$, $t_i \in \R$ with $\lim_{i \to \infty} t_i = 0$, and $(x_i)$, $x_i \in G/H$, such that $\lim_{i \to \infty} \sLog_{\Gamma, t_i}(x_i) = y$. We would like to show that $y$ lies in the image $\overline{\V}_{G/H}$ of the valuation cone in $\R^s$. 
By Proposition \ref{prop-phi-lambda-inequalities} we know there exists a constant $c>0$ such that
$$\phi_{\lambda}(x_i) \phi_{\mu}(x_i) \geq c \phi_{\gamma}(x_i).$$
Taking $\log_{t_i}$ of both sides we have:
$$\log_{t_i}(\phi_{\lambda}(x_i)) + \log_{t_i}(\phi_{\mu}(x_i)) \leq \log_{t_i}(c) +\log_{t_i}(\phi_{\gamma}(x_i)).$$
Now observe that in order to show $y \in \overline{\V}_{G/H}$ it would be enough to find $v \in \mathcal{Q}_{G/H} \otimes \R := \Hom(\Lambda_{G/H}, \R)$ such that for any highest weight $\nu \in  \Lambda_{G/H}$ we have $$\lim_{i \to \infty} \log_{t_i}\phi_\nu(x_i) = \langle v, \nu \rangle.$$
Then by the above inequalities we have $\langle v, \lambda \rangle + \langle v, \mu \rangle \leq \langle v, \gamma \rangle$, for all $\lambda$, $\mu$, $\gamma$ such that $\lambda + \mu - \gamma$ is a tail and hence $v \in \V_{G/H}$.
We expect that the proof that such a $v$ exists would be a variant of the argument in the proof of Theorem \ref{th-phi-approach-inv-valuation}. The case when $s=r$, namely when the highest weight monoid is freely generated, might be easier to work out. 
\end{Rem}
}

In the next section, we show that in a number of interesting cases, the sets $\sLog_t(G/H)$ not only contain the valuation cone, but in fact they do limit to exactly the valuation cone in the Kuratowski sense. 

\section{Examples}
\label{sec-exs}
In the previous section, we asked in Question~\ref{question: sph-amoeba-approach-sph-trop} 
for conditions under which, in analogy with the classical case, the spherical amoebae approach the spherical tropicalization. 
In this section we focus on Question~\ref{question: sph-amoeba-approach-sph-trop}(1), i.e. the 
case when we take $Y=G/H$. We analyze the following interesting cases of spherical homogeneous spaces: (1) the ``group case'' of $G \times G$ acting on $G$, for $G=\SL_n(\C)$, (2) the basic affine space $G/H = \SL_n(\C)/U$, (3) the case $G/H = \SL_n(\C)/\SL_{n-1}(\C)$, and (4) the case of the ``space of hyperbolic triangles'' (the precise definition is in Section~\ref{subsec: hyperbolic triangles}). In the cases of the basic affine space $\SL_n(\C)/U$ and $\SL_n(\C)/\SL_{n-1}(\C)$, it turns out that the spherical amoebae coincide with the valuation cone for all $t$, so the limit does indeed coincide with the valuation cone, with a rather trivial limiting process. However, for the ``group case'' and the hyperbolic triangles case, the spherical amoeba is different from the valuation cone but the limit does indeed approach the valuation cone via a non-trivial limiting process (see Figure~\ref{fig:g-times-g-mod-g}). 

As in Section~\ref{sec: limit}, throughout this section we work over $\R$.

\subsection{ Two small examples with $G=\SL_2(\C)$  } 
\label{subsec-SL2-case}

\subsubsection{The case of $X = \SL_2(\C)/T = (\mathbb{P}^1 \times \mathbb{P}^1) \setminus \Delta$} 

Let $G = \SL_2(\C)$ and let $H=T$, the maximal (diagonal) torus of $G$.  Consider the action of $G$ on $\mathbb{P}^1$ by left multiplication, and the corresponding diagonal action of $G$ on $\mathbb{P}^1 \times \mathbb{P}^1$.  Then it is not hard to see that the stabilizer of the point $x_0 = ([1:0], [0:1]) \in \mathbb{P}^1 \times \mathbb{P}^1$ is precisely $T$, and that the orbit of $x_0$ under the $G$-action is $\mathbb{P}^1 \times \mathbb{P}^1 \setminus \Delta$, where $\Delta$ denotes the diagonal copy of $\mathbb{P}^1$ in the direct product. The compact group $K$ is $SU(2)$ in this case. The (diagonal) action of $SU(2)$ on $\mathbb{P}^1 \times \mathbb{P}^1$ has moment map 
$$
\Phi([z],[w]) = \frac{i}{2} \left( \frac{zz^*}{\| z \|^2} + \frac{ww^*}{\| w\|^2} \right) - \left( \frac{i}{4 \| z\|^2} \mathrm{tr}(zz^*) +   \frac{i}{4 \| w\|^2} \mathrm{tr}(ww^*) \right) \cdot I_{2 \times 2} \in \mathfrak{su}(2)^* 
$$
where $([z],[w])$ are homogeneous coordinates on $\mathbb{P}^1 \times \mathbb{P}^1$ (so $z,w$ are considered as elements of $\C^2$) and $I_{2 \times 2}$ denotes the $2 \times 2$ identity matrix. Composing this with the map which quotients by the coadjoint action of $SU(2)$ on $\mathfrak{su}(2)^*$, i.e. the so-called ``sweeping map'' $\mathfrak{su}(2)^* \to \R_{\geq 0} = \mathfrak{su}(2)/SU(2)$, yields that the Kirwan map $\mathbb{P}^1 \times \mathbb{P}^1 \to \R_{\geq 0}$ is given by 
$$
([z],[w]) \longmapsto \frac{1}{2} \sqrt{ \left( \frac{ |z_1|^2}{\| z\|^2} + \frac{|w_1|^2}{\| w\|^2} - 1 \right)^2 + \left( \frac{ |z_1|^2 |z_2|^2}{\| z\|^4} + \frac{2 \mathrm{Re}(z_1 \overline{w}_1 \overline{z}_2 w_2)}{ \| z\|^2 \| w \|^2} + \frac{ |w_1|^2 |w_2|^2}{\| w\|^4} \right)}. 
$$
The Kirwan polytope (i.e. the image of $\mathbb{P}^1 \times \mathbb{P}^1$ under the above map) is the interval $[0,\frac{1}{2}]$ and the image of the diagonal is straightforwardly computed to be $\frac{1}{2}$, so the distinguished $K$-orbit (the diagonal) corresponds to the boundary value. 

There is also another natural parametrization of $K$-orbits which can be described in terms of the angle between two complex lines in $\C^2$. More specifically, given $([z],[w]) \in \mathbb{P}^1 \times \mathbb{P}^1$ where $z,w \in \C^2 \setminus \{0\}$, we may define 
$$
\rho([z],[w]) := 1 - \frac{|\langle z,w \rangle|^2}{\|z\|^2 \|w\|^2},
$$ where $\langle z, w \rangle$ is the standard Hermitian product on $\C^2$. Geometrically, $\rho$ is the quantity $\sin^2(\theta)$ where $\theta$ is the angle between the two complex lines spanned by $z$ and $w$, or equivalently, the spherical distance between two distinct points $p, q \in \mathbb{P}^1 = S^2$ if we represent $p$, $q$ by complex vectors $z, w \in \C^2$ with $\| z \| = \| w \| = 1$.  It turns out that $\rho([z],[w]) = \rho([z'],[w'])$ if and only if the two pairs $([z],[w]), ([z'], [w'])$ are in the same $K=\SU(2)$-orbit, so the function $\rho$ provides a parametrization of $K$-orbits in $G/T$. It follows from the Cauchy-Schwarz inequality that the image of $\rho$ is the interval $(0, 1]$. 
Thus, composing $\rho$ with $-\log$, i.e. $([z],[w]) \longmapsto - \log(\rho([z],[w]))$, we obtain an identification of the $K$-orbit space with the valuation cone (which in this case is $\mathbb{R}_{\geq 0}$). 

Finally, there is also our spherical function $\phi_2$ corresponding to the irreducible representation of $\SL_2(\C)$ of highest weight $2$ in $\C[\SL_2(\C)/T]$. This can be computed to be 
$$
\phi_2([z],[w]) = \frac{1}{\lvert z_1 w_2 - z_2 w_1 \rvert^2} \left( \lvert z_1 w_1 \rvert^2 + \frac{1}{2} \lvert z_1 w_2 + w_1 z_2 \rvert^2 + \lvert z_2 w_2 \rvert^2 \right). 
$$
One can show that $\phi_2([z], [w]) \geq 1/2$. On the other hand, $\phi_2((1: 1), (-1: 1)) = 1/2$. This shows that the image of $\phi_2$ is $[1/2, \infty)$. Hence we have a parametrization of the $K$-orbits by the points in this interval. One can obtain a parametrization of the $K$-orbit space by $\R_{\geq 0}$ (the valuation cone) by taking the limit, as $t \to 0$, of the image of $-\log_t(\phi_2)$.

The above computations show that, in this case, we have three parametrizations of the space of $K$-orbits $SU(2) \backslash SL_2(\C)/T$ by (half open) intervals in $\R$, namely: the spherical function $\phi_2$ above, the Kirwan map computed above, and the map $\rho$ as above. 

\subsubsection{The case of $X = \SL_2(\C)/N(T)$} 
{As in the last section, we take $G =\SL_2(\C)$ but this time we take $H=N(T)$. In this situation we have $N(T)/T \cong \Z/2\Z$ so there is a natural map $\SL_2(\C)/T \to \SL_2(\C)/N(T)$. The homogeneous space $\SL_2(\C)/N(T)$ can be identified with $\mathbb{P}^2 \setminus Q$ where $Q$ is a smooth conic, as follows: consider the map $\Sym^1(\C^2) \times \Sym^1(\C^2) \to \Sym^2(\C^2)$ given by multiplication. Note that $\Sym^1(\C^2) \cong \C^2$ and $\Sym^2(\C^2) \cong \C^3$. This product map induces a morphism $(\mathbb{P}^1 \times \mathbb{P}^1) \setminus \Delta \to \mathbb{P}^2 \setminus Q$, where $Q$ is the smooth conic defined by the vanishing of the discriminant on $\Sym^2(\C^2) \cong \C^3$. One sees that the natural projection $\SL_2(\C)/T \to \SL_2(\C) / N(T)$ is then identified with $(\mathbb{P}^1 \times \mathbb{P}^1) \setminus \Delta \to \mathbb{P}^2 \setminus Q$. The non-identity element in the quotient $N(T)/T \cong \Z/2\Z$ corresponds to the involution on $\mathbb{P}^1 \times \mathbb{P}^1$ exchanging the two factors. This involution leaves the Kirwan map, $\rho$ and $\phi_2$ of the previous section invariant. This implies that all of these functions descend to functions on $\SL_2(\C)/N(T)$ which also parametrizes the $K$-orbits on $\SL_2(\C)/N(T)$.  
}

\subsection{The group case}
\label{subsec-ex-gp-case}
As a first example, we consider $G$ equipped with the action of $G \times G$ by left and right multiplication, as described in Example \ref{ex-sph-var}(5). In particular, here we identify $G$ with the homogeneous space $G \times G/G_{\diag}$. We choose the Borel subgroup $B \times B^- \subseteq G \times G$ where $B \subseteq G$ is a Borel subgroup of $G$ and $B^-$ is its opposite. By the Peter-Weyl Theorem, $\C[G]$ decomposes as a $G \times G$-module as follows: 
\[
  \C[G] \cong \bigoplus_{\lambda \in \Lambda^+}  V_\lambda \otimes V_\lambda^* 
\]
where $\Lambda^+$ is the set of dominant weights of $G$. It follows from Definition~\ref{def: Lambda plus} that $\Lambda_G^+ = \Lambda_{G \times G/G_{\diag}}^+ \cong \Lambda^+$ and that 
the lattice $\Lambda_{G}$ of $G \cong G \times G/G_{\diag}$ as a $G \times G$-spherical variety is the sublattice $\{(\lambda, -\lambda) \mid \lambda \in \Lambda \} \subset \Lambda \times \Lambda$. Using Theorem~\ref{th-Sumihiro}, \cite[Theorem 24.2]{Timashev}, \cite[Lemma 24.3]{Timashev} and the Cartan decomposition, it can then be shown that the valuation cone $\V_{G}$ can be identified with the antidominant Weyl chamber in $\check{\Lambda}$, i.e., $\V_G = \{ \xi \in \check{\Lambda} \colon \langle \xi, \alpha \rangle \le 0 \; \text{for all positive roots} \; \alpha \}$.  

In the case of $G=\SL_n(\C)$, to which we now restrict, we can also give
an explicit description of the spherical functions. 
We choose the basis $\{(\omega_i, -\omega_i)\}_{1 \leq i \leq n-1}$ for $\Lambda_{\SL_n(\C)}$ where $\omega_i$ denotes the usual $i$-th fundamental weight of $\SL_n(\C)$. The corresponding irreducible $\SL_n(\C) \times \SL_n(\C)$-representation is $\Lambda^i \C^n \otimes (\Lambda^i \C^n)^*$. To construct the spherical function $\phi_{(\omega_i, - \omega_i)}$ corresponding to $(\omega_i, -\omega_i)$, we need to find an isomorphic copy of $\Lambda^i \C^n \otimes (\Lambda^i \C^n)^*$ in $\C[\SL_n(\C)]$, considered as an $\SL_n(\C) \times \SL_n(\C)$-representation. Let $J = \{j_1 < j_2 < \cdots < j_i\}, K = \{k_1 < k_2 < \cdots < k_i\}$ be subsets of $\{1,2,\ldots,n\}$ of cardinality $i$. Let $p_{J,K}$ denote the $(i \times i)$ minor of an $n \times n$ matrix corresponding to the subsets $J$ and $K$, i.e., the determinant of the submatrix with rows in $J$ and columns in $K$. We view $p_{J,K}$ as an element in $\C[\SL_n(\C)]$. By explicitly analyzing the action of the Chevalley generators $E_\alpha, F_\alpha$ (for positive roots $\alpha$) as well as by computing the $T$-weights of the $p_{J,K}$, it is not difficult to see that the span of the $p_{J,K}$ is a subrepresentation of $\C[\SL_n(\C)]$ isomorphic to $\Lambda^i \C^n \otimes (\Lambda^i \C^n)^*$. Moreover, it is straightforward to see that each $p_{J,K}$ is a $T$-weight vector, and that there exists a $K$-invariant inner product with respect to which the $p_{J,K}$ form an orthonormal basis. Thus we have 
\[
\phi_{(\omega_i, -\omega_i)}(x) = \frac{1}{\binom{n}{i}} \sum_{J, K} \lvert p_{J,K} \rvert^2
\]
where the normalization factor ensures that $\phi_{(\omega_i, - \omega_i)}(x) = 1$. 

Recall the singular value decomposition theorem, which states that any  $A \in \SL_n(\C)$ can be expressed as a product
\[
  A = U_1 D U_2
\]
where $U_1, U_2 \in \SU_n$ and $D$ is a diagonal matrix with positive real entries whose product is equal to $1$. Let $\Gamma = \{(\omega_i, -\omega_i) \, \mid \, 1 \leq i \leq n-1 \}$ and consider the spherical logarithm map $\sLog_{\Gamma, t}$ on $\SL_n(\C)$ corresponding to $\Gamma$. Since the spherical function $\phi_i$ is $(\SU_n \times \SU_n)$-invariant by construction, it follows that
\[
  \sLog_{\Gamma, t}(\SL_n(\C)) = \left\{ \left( \sum_{|I|=i} \prod_{k \in I} d_k \right) \in \R^{n-1} \, \mid \, 1 \leq i \leq n-1 ;  d_1, \ldots, d_n \in \R_+, \prod_{k=1}^n d_k = 1 \right\} \text{.}
\]
(Note that the components of the function $\log_t \circ \Phi$ are exactly the symmetric functions on the eigenvalues of $A$, i.e., the coefficients of the characteristic polynomial of $A$.) 

For the special case $n = 3$, an elementary calculus (maximization) computation shows that the resulting region is bounded by
\begin{multline*}
  \left\{ \left( \log_t\left(\frac{2x^3+1}{3x^2}\right), \log_t\left(\frac{2+x^3}{3x}\right) \right) \colon x \in [1, \infty)\right\} \\
  \text{and} \; \left\{ \left(\log_t\left(\frac{2+x^3}{3x}\right), \log_t\left(\frac{2x^3+1}{3x^2}\right) \right) \colon x \in [1, \infty)\right\}
\end{multline*}
from which it can be seen that as $t$ approaches $0$, the spherical amoebae approach $\overline{\V}_G$. 
See Figure \ref{fig:g-times-g-mod-g}. {Also compare with \cite[Example 7.7]{Kaveh-Manon}.}

\begin{figure}[!ht]
  \centering
  \includegraphics{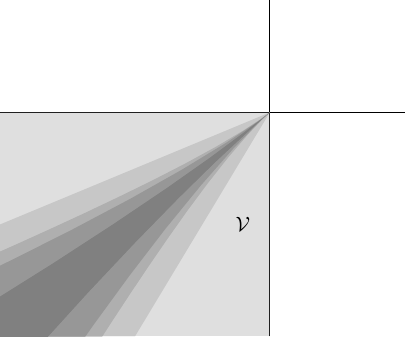}
  \caption{The images of the spherical logarithm $\mathrm{im} (\log_t(\Phi))$ approach the valuation cone as $t \to 0$ for the case $(\SL_3(\C) \times \SL_3(\C))/\SL_3(\C)$.}
  \label{fig:g-times-g-mod-g}
\end{figure}

\subsection{The basic affine space \texorpdfstring{$\SL_n(\C)/U$}{SL(n,C)/U}}
\label{subsec-ex-G/U}
Consider the spherical homogeneous space $G / U$ where $U$ is a maximal unipotent subgroup (see Example \ref{ex-sph-var}(4)). One knows that $G / U$ is a quasi-affine variety. The ring of regular functions $\C[G / U]$ can be identified with the ring $\C[G]^U$ of (right) $U$-invariants. Thus, as a $G$-module it decomposes as:
\[
  \C[G] = \bigoplus_{\lambda \in \Lambda^+} V_\lambda^* \text{.}
\]
Moreover, the multiplication in the algebra $\C[G/U]$ corresponds to Cartan multiplication between the $V_\lambda^*$. It follows that the tail cone consists of the origin only and hence the valuation cone $\V_{G/U}$  is the whole vectors space $\mathcal{Q}_{G/U} = \Lambda_\Q$. 

Let us give an explicit description of the spherical functions on $G/U$ in the case of $G = \SL_n(\C)$. One can take $U$ to be the subgroup of upper triangular matrices with $1$'s on the diagonal. Set $V = \C^n$ and consider the natural action of $\SL_n(\C)$ on the exterior algebra $\Lambda^* V$. If $e_1, \ldots, e_n$ are the standard basis vectors of $V$, then the stabilizer of $x_0 \coloneqq (e_1, e_1\wedge e_2, \ldots, e_1 \wedge \ldots \wedge e_n)$ is $U$. As $\SL_n(\C) \to \Lambda^* V$ is a locally closed embedding with quasi-affine image, we obtain a surjective map $\C[\Lambda^* V] \to \C[\SL_n(\C)/U]$ and it straightforwardly follows that the algebra $\C[\SL_n(\C)]^U$ is generated by the so-called flag minors: Let $x \in \SL_n(\C)$. For subsets $I, J \subset \{1, \ldots, n\}$ with $|I| = |J| = i$, let $p_{I,J}(x) $ denote the minor of $x$ which is the determinant of the $i \times i$ submatrix of $x$ with rows corresponding to $I$ and columns corresponding to $J$. A \emph{flag minor} is a minor of the form $p_{I, \{1, \ldots, i\}}$, that is when $J = \{1, \ldots, i\}$ corresponds to the first $i$ columns of $x$.

Let $\Gamma = \{\omega_1, \ldots, \omega_{n-1}\}$. There are $n-1$ spherical functions corresponding to the fundamental weights $\omega_1, \ldots, \omega_{n-1}$. The representation $V_{\omega_i}$ corresponding to the $i$-th fundamental weight is $\bigwedge^i V$ and the $i \times i$ flag minors are a vector space basis consisting of weight vectors.
{Consider the standard Hermitian product on $V$. This induces Hermitian products on $\bigwedge^i V$ for al $i$. One verifies that, for each $i$, the vector $\{e_{j_1} \wedge \cdots e_{j_i} \mid 1 \leq j_1 < \cdots < j_i \leq i \}$ is an orthonormal basis for $\bigwedge^i V$. It follows that the flag minors are an orthonormal basis for $\bigwedge^i V \subset \C[\SL_n(\C)]^U$.}
Thus the $i$-th spherical function $\phi_i$ corresponding to $\omega_i$ is given by:
\[
  \phi_i(x) = \sum_{I} |p_{I, \{1, \ldots, i\}}(x)|^2 \text{,}
\]
where the sum is over all subsets $I \subset \{1, \ldots, n\}$ with $|I| = i$. Note that $\phi_i(eU) = 1$.  Consider the spherical logarithm map $\sLog_{\Gamma, t}$ corresponding to our choice of $\Gamma$. We have
\[
  x=\begin{psmallmatrix}
    t_1 & & & &\\
    &t_1^{-1} t_2^{\mbox{}} & & &\\
    &&\ddots& &\\
    &&&t_{n-2}^{-1} t_{n-1}^{\mbox{}} &\\
    &&&& t_{n-1}^{-1}
  \end{psmallmatrix} \in \SL_n(\C)
\]
with $\sLog_{\Gamma, t}(x) = (\log_t|t_1|^2, \ldots, \log_t|t_{n-1}|^2)$, and thus the image of $\sLog_{\Gamma, t}$ is the whole space $\R^{n-1}$ and agrees (for all $t$) with the valuation cone. 

\subsection{Example of \texorpdfstring{$\SL_n(\C) / \SL_{n-1}(\C)$}{SLn/SL(n-1)}}
\label{subsec:sln-sln-1}

Throughout this section, we let $G = \SL_n(\C)$ and $H = \SL_{n-1}(\C)$ for  $n \geq 3$. {The variety $G/H$ is a very well-known example of a spherical homogeneous space. Nevertheless, for the convenience of the reader we try to give details of the proofs in this section.} 

The maximal compact subgroup of $\SL_n(\C)$ is $K \coloneqq \SU_n$ which we realize explicitly as 
\[
  \SU_n = \set{ A \in \C^{n \times n} \with \overline{A}^t A = I_n} 
\]
where $\overline{A}^t$ denotes the conjugate transpose. 
Let $\SL_n(\C)$ act on $\C^n \times \C^n$ as follows: 
\begin{equation}\label{eq: SLn action}
  \SL_n(\C) \times \rleft( \C^n \times \C^n \rright) \to \C^n \times \C^n;  (A, \vv, \ww) \mapsto ( A\vv, \prescript{t}{}{A}^{-1}\ww) \text{.}
\end{equation}
We denote the standard basis of $\C^n$ by $\ee_1, \ldots, \ee_n$ and the coordinates on the two $\C^n$ factors by $x_1, \ldots, x_n$ and $y_1, \ldots, y_n$ respectively. With respect to the action~\eqref{eq: SLn action} it is straightforward to compute that the stabilizer of $(e_1, e_1) \in \C^n \times \C^n$ is 
\[
  \stab_{\SL_n(\C)} (\ee_1,\ee_1) = \set{%
  \tikz[baseline=(M.west),every left delimiter/.style={xshift=.4em},every right delimiter/.style={xshift=-.4em}]{%
    \node[matrix of math nodes,matrix anchor=west,left delimiter=(,right delimiter=),ampersand replacement=\&,inner sep=0pt,nodes={inner sep=3.5pt},row sep=0pt] (M) {%
      1 \& 0 \& \ldots \& 0 \\
      0 \& \mbox{}  \& \mbox{} \& \mbox{} \\
      \smash{\vdots}\vphantom{\int\limits^x} \& \mbox{} \& \mbox{} \& \mbox{} \\
      0 \& \mbox{} \& \mbox{} \& \mbox{} \\
    };
    \node[draw,fit=(M-2-2)(M-4-4),inner sep=-1pt,label=center:$\SL_{n-1}$] {}; } } \cong \SL_{n-1}(\C) \text{.}
\]

\begin{Lem}
\[
  \SL_n(\C) \cdot (\ee_1, \ee_1) = \set{ (\xx, \yy) \in \C^n \times \C^n \with \sum_{i=1}^n x_i y_i = \mathbf{y}^t \mathbf{x} = 1} \subseteq \C^n \times \C^n 
\]
where we use the notation $\mathbf{y}:=(y_1,\ldots, y_n)$ and $\mathbf{x} := (x_1,\ldots, x_n)$.
\end{Lem} 

\begin{proof} 
 The inclusion ``$\subseteq$'' is straightforward, since for any $A \in \SL_n(\C)$, it immediately follows from the definition of the action that $A \cdot (e_1, e_1) := (\mathbf{x} = Ae_1, \mathbf{y} = \prescript{t}{}{A}^{-1} e_1)$ would have the property $\mathbf{y}^t \mathbf{x} = e_1^t (\prescript{t}{}{A}^{-1})^t A e_1 = e_1^t A^{-1} A e_1 = e_1^t e_1 = 1$. 
 
 To see the inclusion ``$\supseteq$'', suppose that $(\mathbf{x}, \mathbf{y}) \in \C^n \times \C^n$ satisfies $\mathbf{y}^t \mathbf{x} = 1$. We wish to find an element in $\SL_n(\C)$ that takes $(e_1,e_1)$ to $(\mathbf{x}, \mathbf{y})$. First notice that since $\mathbf{y} \neq 0$, there must exist some matrix $g = (a_{ij})  \in \SL_n(\C)$ with the property that the top row of $g$ is $(y_1,y_2,\ldots,y_n)$. Now define a matrix $g'$ by replacing the first column of $g^{-1}$ with $(x_1,\ldots,x_n)^t$ (a column vector). We claim that $g'$ is an element in $\SL_n(\C)$ and that $g' \cdot (e_1,e_1) = (\mathbf{x}, \mathbf{y})$. By construction we have $ge_1=\mathbf{x}$, so it remains to see that $g' \in \SL_n(\C)$ and that $\prescript{t}{}{g}^{-1} e_1 = \mathbf{y}$. 
 
 We claim $g' \in \SL_n(\C)$. To see that $\det(g')=1$ we make an explicit computation. Given an $n \times n$ matrix $A$ we denote by $A_{i,j}$ the $(n-1) \times (n-1)$ matrix obtained by deleting its $i$-th row and $j$-th column. Below we compute $\deg(g')$ by expansion along its left-most column, which by construction is $\mathbf{x}$. 
 \begin{equation*}
 \begin{split} 
 \det(g') & = \sum_{i=1}^n (-1)^{i+1} x_i \cdot \det(g'_{i,1}) \\
 & = \sum_{i=1}^n (-1)^{i+1} x_i \cdot \det((g^{-1})_{i,1}) \textup{ since $g'$ is the same as $g^{-1}$ except for the leftmost col} \\
 & = \sum_{i=1}^n x_i a_{1,i} \textup{ by Cramer's rule for matrix inverses and since $(g^{-1})^{-1} = g$} \\
 & = \sum_{i=1}^n x_i y_i \textup{ since the top row of $g$ is $\mathbf{y}$ by construction} \\
 & = 1 \textup{ by assumption on $\mathbf{x}$ and $\mathbf{y}$.} \\
 \end{split} 
 \end{equation*}
 We next claim that $\prescript{t}{}{(g')}^{-1} e_1 = \mathbf{y}$, or in other words, that $(g')^{-1}$ has top row 
 equal to $\mathbf{y}$. By a Cramer's rule argument similar to that above, we can see that the $(1,i)$-th matrix entry of $(g')^{-1}$ is the $(i,1)$-th cofactor of $g'$, which is equal to the $(i,1)$-th cofactor of $g^{-1}$, which in turn is equal to the $(1,i)$-th entry of $(g^{-1})^{-1}=g$.  But this last value is exactly $y_i$ by construction of $g$, so we are done. 
 \end{proof} 
 
 It follows from the above lemma that the homogeneous space $\SL_n(\C)/\SL_{n-1}(\C)$ is an affine variety and its coordinate ring may be identified as 
\[
  \C \rleft[ \SL_n(\C) / \SL_{n-1}(\C) \rright] \cong \C \rleft[ x_1, \ldots, x_n, y_1, \ldots, y_n \rright] / \rleft\langle \sum_{i=1}^n x_i y_i = 1 \rright\rangle \text{.}
\]
Let $B$ be the Borel subgroup of upper triangular matrices in $\SL_n(\C)$ and $T \subseteq B$ the maximal torus consisting of all diagonal matrices in $\SL_n(\C)$. Let $U$ be the maximal unipotent subgroup of $B$ consisting of upper triangular matrices with $1$'s on the diagonal. We show next that the homogeneous space $\SL_n(\C)/\SL_{n-1}(\C)$ is in fact spherical, i.e., has a dense open $B$-orbit. We have the following. 

\begin{Lem}\label{lemma: open B orbit for SLn/SLn-1}
The subset 
  \[
    B \cdot (\ee_1 + \ee_n,\ee_1) = \set{(\xx,\yy) \in \C^n \times \C^n \with \yy^t \xx = 1 \textup{ and } x_ny_1 \neq 0}  \text{.}
  \]
is a dense open $B$-orbit in $\SL_n(\C)/\SL_{n-1}(\C)$. In particular, $\SL_n(\C)/\SL_{n-1}(\C)$ is a spherical homogeneous space. 
\end{Lem}

\begin{proof}
  The inclusion ``$\subseteq$'' is straightforward. To show the reverse inclusion, 
  let $(\xx, \yy) \in \C^n \times \C^n$ with $\sum_{i=1}^n x_i y_i = 1$ and $x_ny_1 \neq 0$. Consider the matrix 
  \[
    A \coloneqq \begin{pmatrix}
      \tfrac{1}{y_1} & \tfrac{-y_2}{y_1} & \ldots & \tfrac{-y_{n-2}}{y_1} & \tfrac{-y_{n-1}}{x_n} & x_1 - \tfrac{1}{y_1} \\
      & 1 & 0 & \ldots & 0 & x_2\\
      & & \ddots & \ddots & \vdots & \vdots\\
      & & & 1 & 0 & x_{n-2}\\
      & & & & \tfrac{y_1}{x_n} & x_{n-1}\\
      & & & & & x_n
    \end{pmatrix} \in B \text{.}
  \]
  We claim that $A$ is in the Borel subgroup $B$ of $\SL_n(\C)$ and that $A \cdot (\ee_1 + \ee_n, \ee_1) = (\xx,\yy)$. By construction $A$ is upper-triangular, and from the diagonal entries it easily follows that $\det(A)=1$, so $A$ is in $B$. It is straightforward to see from the definition of $A$ that $A(\ee_1 + \ee_n) = \xx$. To see that 
  $\prescript{t}{}{A}^{-1} \ee_1 = \yy$, it suffices to see that the top row of $A^{-1}$ is $\yy$. Moreover, since $\sum_i x_i y_i=1$, the last entry $y_{n}$ is determined by $\xx$ and $y_1,\ldots,y_{n-1}$, it in fact suffices to check the first $n-1$ entries. Fix $j$ with $1 \leq j \leq n-1$. We wish to show that the $(1,j)$-th entry of $A^{-1}$ is equal to $y_j$. It is straightforward to check this explicitly using the adjoint form of $A$ as in the proof of the previous lemma and the explicit formula for $A$ given above. 
\end{proof}

We now describe parts of the Luna-Vust data associated to the spherical homogeneous space $G/H = \SL_n(\C)/\SL_{n-1}(\C)$. Recall that the \textbf{colors} $\mathcal{D}$ of $G/H$ is defined to be the set of $B$-invariant prime divisors of $G/H$. From Lemma~\ref{lemma: open B orbit for SLn/SLn-1} we know the open $B$-orbit is defined by the condition $x_n y_1 \neq 0$, so any $B$-invariant prime divisor is contained in its complement $\{x_n=0 \textup{ or } y_1 = 0\} = \{x_n=0\} \cup \{y_1 = 0\}$. 

\begin{Lem}\label{lemma: colors xn and y1} 
 The two elements $x_n, y_1$ are prime elements in $\C[\SL_n/\SL_{n-1}]$.
\end{Lem}

\begin{proof}
  To see that $x_n$ is a prime element, it would suffice to show that $\C[\SL_n/\SL_{n-1}]/\rleft \langle x_n \rright \rangle$ is a domain. Note that there is an isomorphism
  \[
    \C[\SL_n/\SL_{n-1}] / \rleft \langle x_n \rright \rangle = \C[x_1, \ldots, x_{n-1},y_1, \ldots, y_n] / \rleft\langle x_1 y_1 + \ldots + x_{n-1} y_{n-1} - 1 \rright\rangle 
  \]
  so it suffices to show that the RHS is a domain. Since the polynomial $x_1 y_1 + \ldots + x_{n-1} y_{n-1} -1$ is of degree $2$, if it factors non-trivially then it must be of the form $x_1 y_1 + \ldots + x_{n-1} y_{n-1} -1 = g \cdot h$ where $\deg(g)=\deg(h)=1$ so both $g$ and $h$ are linear polynomials in $x_1,\ldots,x_{n-1}, y_1,\ldots,y_{n-1}$ (with constant term). It is not hard to see by direct computation that this is impossible, so $x_1 y_1 + \ldots + x_{n-1}y_{n-1}-1$ is irreducible and hence the RHS of the equality above is an integral domain. We conclude $x_n$ is prime in $\C[\SL_n/\SL_{n-1}]$, as desired. The argument for $y_1$ is similar. 
  \end{proof}

From Lemma~\ref{lemma: colors xn and y1} it follows that $D_1 = \{x_n=0\}$ and $D_2 = \{y_1 = 0\}$ are prime divisors. Moreover, it is immediate from~\eqref{eq: SLn action} that both are $B$-invariant. It follows that $D_1, D_2$ are the two colors in $\SL_n(C)/\SL_{n-1}(\C)$, i.e., $\mathcal{D} = \{D_1, D_2\}$. 

\begin{Rem}
  Note that $(\ee_1,\ee_1)$ is not contained in the open $B$-orbit and the colors consist of several $B$-orbits.
  Indeed, we have the following ``interesting'' $B$-orbits:
  \begin{itemize}
  \item $(\ee_1 + \ee_n,\ee_1)$ is a base point of the open $B$-orbit in $\SL_n/\SL_{n-1}$.
  \item $B \cdot (\ee_1 + \ee_{n-1}, \ee_1)$ is the open $B$-orbit of the color $D_1 = \set{x_n=0}$.
  \item $B \cdot(\ee_1 + \ldots + \ee_n, \ee_2)$ is the open $B$-orbit of the color $D_2 = \set{y_1=0}$.
  \end{itemize}
  Note that the colors $D_1,D_2$ are each a union of multiple $B$-orbits.
\end{Rem}

Next, we need to compute the $B$-semi-invariants in the ring $\C[\SL_n(\C)/\SL_{n-1}(\C)]$. We need 
some preparation. 

\begin{Lem}\label{lemma: units in SLn/SLn-1}
The coordinate ring $\C[\SL_n(\C)/\SL_{n-1}(\C)]$ is a UFD, a unique factorization domain. 
Moreover, the units $\C[\SL_n(\C)/\SL_{n-1}(\C)]^*$ in $\C[\SL_n(\C)/\SL_{n-1}(\C)]$ consists of the non-zero constants, i.e. $\C[\SL_n(\C)/\SL_{n-1}(\C)]^* = \C^*$.
\end{Lem}

\begin{proof}
 Consider the localization of $\C[\SL_n/\SL_{n-1}]$ by the prime element $x_n$ and observe that we have an isomorphism
  \[
    \C[\SL_n/\SL_{n-1}]_{x_n} \cong \C[x_1, \ldots, x_n, y_1, \ldots, y_{n-1},x_n^{-1}]
  \]
  since $y_n = (1-\sum_{i=1}^{n-1}x_iy_i)/x_n$. In particular, the ring on the right hand side is a UFD. 
From \cite[Lemma 19.20]{Eisenbud} we know that if $R$ is an integral domain and $p$ a prime element and the localization $R_p$ is a UFD, then $R$ is a UFD.  Applying this to $R=\C[\SL_{n}(\C)/\SL_{n-1}(\C)]$ and $p=x_n$, we conclude that $\C[\SL_n/\SL_{n-1}]$ is also a UFD. 

Now we wish to prove the claim about the units. Embed $\C[\SL_n(\C)/\SL_{n-1}(\C)]$ into the localization 
$\C[\SL_n(\C)/\SL_{n-1}(\C)]_{x_n} \cong \C[x_1,\ldots,x_n, x_n^{-1}, y_1, \ldots, y_{n-1}]$. In particular, any unit in $\C[\SL_n(\C)/\SL_{n-1}(\C)]$ must also be a unit in the localization. The units of $\C[x_1,\ldots,x_n, x_n^{-1}, y_1, \ldots, y_{n-1}]$ are of the form $c x_n^{k}$ for $c \in \C^*$ and $k \in \Z$. However, $x_n$ is a prime element in $\C[\SL_n(\C)/\SL_{n-1}(\C)]$, so it is not a unit in $\C[\SL_n(\C)/\SL_{n-1}(\C)]$. Thus the only units in $\C[\SL_n(\C)/\SL_{n-1}(\C)]$ are the non-zero constants.  
\end{proof}

Recall that a function $f \in \C[\SL_n(\C)/\SL_{n-1}(\C)]$ is called $B$-semi-invariant if there exists a character $\chi_f \in \mathcal{X}(B)$ such that $b \cdot f = \chi_f(b) f$ for all $b \in B$. The set of $B$-semi-invariants is denoted $\C[\SL_n(\C)/\SL_{n-1}(\C)]^{(B)}$. From the above lemma we can deduce the following.

\begin{Cor}
  \label{cor:sln-sln-1-b-semi}
We have
  \[
    \C \rleft[ \SL_n(\C)/ \SL_{n-1}(\C) \rright]^{(B)} = \set{ x_n^k y_1^\ell \with k,\ell \in \Z_{\ge0}} \text{.}
  \]
\end{Cor}

\begin{proof}
  Let $f \in \C[ \SL_n/\SL_{n-1}]^{(B)}$ be non-zero and $B$-semi-invariant. Then the vanishing locus of $f$ must be contained in $\{x_n y_1 = 0\}$ since $f$ cannot vanish on the dense open $B$-orbit $\{x_n y_1 \neq 0\}$ (see Lemma~\ref{lemma: open B orbit for SLn/SLn-1}).  Thus $\Div(f) = k_1 D_1 + k_2 D_2$ for some non-negative integers $k_1$ and $k_2$ and $f/(x_n^{k_1}y_1^{k_2})$ is an invertible regular function on $\SL_n/\SL_{n-1}$. In particular, it is a unit in $\C[\SL_n(\C)/\SL_{n-1}(\C)]$. We saw in Lemma~\ref{lemma: units in SLn/SLn-1} that the units are the non-zero constants, so we conclude that $f = c x_n^{k_1} y_1^{k_2}$ for some $c \in \C^*$, as desired. 
\end{proof}

We can now compute $\Lambda^+_{G/H}$. Let $\omega_1, \ldots, \omega_{n-1}$ be the fundamental weights associated to our choice $(B,T)$, i.e., for $i = 1, \ldots, n-1$ we have 
\[
  \omega_i \begin{pmatrix} b_{11} & b_{12} & \ldots & b_{1n} \\ & b_{22} & \ldots & b_{2n} \\ && \ddots & \vdots \\ &&& b_{nn} \end{pmatrix} = \prod_{j=1}^i b_{jj} = b_{11} \cdots b_{ii} \text{.}
\]
It is straightforward to compute that 
the $B$-weights of $x_n$ and $y_1$ are $\chi_1 \coloneqq \omega_{n-1}$ and $\chi_2 \coloneqq \omega_1$ respectively.  The association $f \mapsto \chi_f$ gives a map $\C[\SL_n(\C)/\SL_{n-1}(\C)]^{(B)} \to \mathcal{X}(B)$ whose image is denoted $\Lambda^+_{G/H} = \Lambda^+_{\SL_n/\SL_{n-1}}$. If $f, g \in \C[\SL_n(\C)/\SL_{n-1}(\C)]^{(B)}$ then $\chi_{f \cdot g} = \chi_f + \chi_g$ so $\Lambda^+_{\SL_n/\SL_{n-1}}$ is a semigroup. In our case $G/H = \SL_n(\C)/\SL_{n-1}(\C)$, the above computation shows that $ \Lambda_{G/H}^+$ is freely generated by $\chi_1$ and $\chi_2$, i.e. $\Lambda_{G/H}^+ \cong \Z_{\ge0} \chi_1 \oplus \Z_{\ge0} \chi_2 \cong \Z_{\ge0}^2$. 
For what follows, we choose $\Gamma = \{\chi_1, \chi_2\}$, with associated $B$-semi-invariant functions $x_n, y_1$. Note that $\Gamma$ is also a basis of $\Lambda_{G/H} \otimes \R$, and with respect to this choice of basis, the map $\rho$ of~\eqref{eq: def rho} becomes 
\[
\rho:   \mathcal{V}_{\SL_n(\C)/\SL_{n-1}(\C)} \hookrightarrow \mathcal{Q},~ \nu \mapsto (\nu(x_n),\nu(y_1)). 
\]
We have the following, where we temporarily work over $\Q$, the most natural setting for this lemma. 

\begin{Lem}
The image of $\mathcal{V}_{\SL_n(\C)/\SL_{n-1}(\C)}$ under $\rho$ is 
  \begin{equation}\label{eq: image of rho}
    \rho(\mathcal{V}_{\SL_n(\C)/\SL_{n-1}(\C)}) = \set{(q_1,q_2) \in \Q^2 \with q_1+q_2 \le 0} \text{.}
  \end{equation}
\end{Lem}

\begin{proof}
We first show that the LHS is contained in the RHS.
  Let $\nu \in \mathcal{V}(G/H)$. For any $i$ with $1 \leq i \leq n$, there is a permutation $g \in \SL_n(\C)$ of the coordinates such that $g \cdot x_n = x_i$. A similar statement holds for $y_1$ and $y_i$. Since $\nu$ is $\SL_n(\C)$-invariant, it follows that
 \[
    \nu(x_1) = \ldots = \nu(x_n) \qquad \text{and} \qquad \nu(y_1) = \ldots = \nu(y_n) \text{.}
  \]
  On the other hand, since $\sum_{i=1}^n x_i y_i =1$ in $\C[\SL_n/\SL_{n-1}]$, we obtain from the axioms of valuations that 
  \[
  \begin{split} 
    0 = \nu(1) = \nu(x_1y_1 + \ldots x_ny_n) & \ge \mathrm{min} \{\nu(x_1 y_1), \ldots, \nu(x_n y_n)\} \\
    & = 
    \mathrm{min} \{ \nu(x_1)+\nu(y_1), \ldots, \nu(x_n)+\nu(y_n)\} = \nu(x_1)+\nu(y_1) \\
    & = \nu(x_n) + \nu(y_1)\text{.}\\
    \end{split}
  \]
  Hence $\rho(\nu)$ lies in the RHS of~\eqref{eq: image of rho}, as desired. 
  
  We now claim that the RHS is contained in the LHS.  By Theorem~\ref{th-val-cone-co-simplicial} the image $\rho(\mathcal{V}_{\SL_n/\SL_{n-1}})$ is a (rational) polyhedral cone, in order to show that the RHS is contained in the LHS, it suffices to show that \textit{integral} points in the RHS are contained in the LHS. 
  Let $(n_1, n_2) \in \Z^2$ with $n_1+n_2 \le 0$. Consider the subvariety 
  \[
  W \coloneqq \set{ u_1v_1 + \ldots + u_nv_n - w^{-n_1-n_2} = 0 } \subseteq \C^n \times \C^n \times \C
  \]
 and consider the $\SL_n(\C)$-action on $\C^n \times \C^n \times \C$ given as the product of the given $\SL_n(\C)$-action on $\C^n \times \C^n$ and the trivial action on the last $\C$ factor.  This induces an $\SL_n(\C)$-action on $W$: indeed, if $\uu^t I_n \vv = w^{-n_1-n_2}$, then for $A \in \SL_n$, we have $(A\uu)^t I_n (A^t)^{-1} \vv = w^{-n_1-n_2}$. Note that $D := \set{ w=0}$ is $\SL_n(\C)$-invariant, so $W \setminus \set{w=0}$ is also invariant. 
  There is an $\SL_n$-equivariant projection
\[
  W \setminus \set{w=0} \to \SL_n/\SL_{n-1}; (\uu,\vv,w) \mapsto (w^{n_1}\uu,w^{n_2}\vv) 
\]
which induces an inclusion of the corresponding function fields $\C(SL_n/\SL_{n-1}) \hookrightarrow \C(W)$. 
We claim that $D$ is a prime divisor in $W$. Indeed, $z \in \C[W]$ is a prime element, since 
\[
  \C[u_1, \ldots, u_n, v_1, \ldots, v_n, w]/\rleft\langle \sum_{i=1}^n u_iv_i - w^{-n_1-n_2}, w \rright\rangle \cong \C[u_1, \ldots, u_n, v_1, \ldots, v_n] / \rleft\langle \sum_{i=1}^n u_iv_i \rright \rangle 
\]
and the ring on the RHS can be seen to be an integral domain from the fact that $\sum_{i=1}^n u_i v_i$ is irreducible in $\C[u_1,\ldots,u_n, v_1,\ldots,v_n]$ by an argument similar to the proof of Lemma~\ref{lemma: colors xn and y1}. 
Thus we may define $\nu_D$ to be the geometric valuation on $\C(W)$ given by the order of vanishing of a rational function along the (prime) divisor $D$. Since $D$ is $\SL_n$-invariant, this is an $\SL_n$-invariant valuation, and the restriction $\nu_D|_{\C(\SL_n/\SL_{n-1})}$ to (the image of) $\C(\SL_n(\C)/\SL_{n-1}(\C))$ yields an $\SL_n$-invariant valuation in $\mathcal{V}_{\SL_n/\SL_{n-1}}$. We compute
\[
  \nu_D(x_n) = \nu_D(u_nw^{n_1}) = n_1 \qquad \text{and} \qquad \nu_D(y_1) = \nu_D(v_1w^{n_2}) = n_2 \text{.}  
\]
Thus we have shown that $(n_1, n_2)$ lies in the image under $\rho$ of $\mathcal{V}_{\SL_n/\SL_{n-1}}$, as desired.
\end{proof}

As discussed above, in order to compare the valuation cone with the image of the spherical logarithm we should tensor with $\R$ and consider $\overline{\V}_{\SL_n/\SL_{n-1}}$. It is clear that the closure is $\{(q_1, q_2) \in \R^2: q_1 + q_2 \leq 0\}$ and this is what we consider as the valuation cone in Proposition~\ref{prop: limit in SLn SLn-1 case} below. 

The above discussion can be summarized in the following picture.  The gray half-space defined by $q_1+q_2 \leq 0$ is the valuation cone $\overline{\mathcal{V}}_{\SL_n/\SL_{n-1}}$. We have also indicated the images under $\rho$ of the (geometric valuations corresponding to the) two colors $D_1$ and $D_2$; note that these do not lie in the valuation cone since they are not $G$-invariant (only $B$-invariant). 

\begin{center}
  \begin{tikzpicture}[scale=0.5]
  
    \clip (-2.25,-2.25) -- (2.25,-2.25) -- (2.25,2.25) -- (-2.25,2.25) -- cycle;
      
    \fill[fill=gray!25] (3,-3) -- (-3,3) -- (-3,-3) -- cycle;
    \node at (-1,-1) {$\mathcal{V}$};
  
    \draw (-3,0) -- (3,0);
    \draw (0,-3) -- (0,3);
    
    \draw[-latex,very thick] (0,0) -- (1,0) node[below] {\tiny $\rho(v_{D_1})$};
    \draw[-latex,very thick] (0,0) -- (0,1) node[left] {\tiny $\rho(v_{D_2})$};
  \end{tikzpicture}
\end{center}

Next we compute the spherical functions. 
Consider the decomposition $\C[\SL_n(\C)/\SL_{n-1}(\C)] \cong \bigoplus_{\lambda \in \Lambda_{G/H}^+} V_\lambda$. Let $V = \C^n$ the natural representation of $\SL_n(\C)$ and $V^*$ its dual representation. Then $V$ has highest weight $\omega_1$ while the dual representation $V^*$ has highest weight $\omega_{n-1}$.  The standard scalar product turns the standard bases into unitary bases on $V$ and $V^*$ respectively. Note that $V$ and $V^*$ are naturally embedded in $\C[\SL_n(\C) / \SL_{n-1}(\C)]$ by the identifications $V \cong \langle y_1, \ldots, y_n \rangle$ and $V^* = \langle x_1, \ldots, x_n\rangle$. The variables $x_1, \ldots, x_n$ and $y_1, \ldots, y_n$ form weight bases (with respect to $T$) of $V^*$ and $V$ respectively. Hence, we can define, following~\eqref{eq: def spherical} in Section~\ref{sec: spherical functions}, the following spherical functions: 
\begin{equation}\label{eq: sphericals for SLn/SLn-1} 
  \phi_{\omega_{n-1}} = |x_1|^2 + \ldots + |x_n|^2  \textup{ and }  \phi_{\omega_1} = |y_1|^2 + \ldots + |y_n|^2 
  \end{equation}
both of which lie in $\rleft( \C \rleft[ \SL_n(\C)/\SL_{n-1}(\C) \rright] \otimes \overline{\C \rleft[ \SL_n(\C)/\SL_{n-1}(\C) \rright]} \rright)^{\SU_n}$.

\begin{Rem}
Note that $\phi_{\omega_{n-1}}$ and $\phi_{\omega_1}$ are both $\SU_n$-invariant, 
because $x_1,\ldots,x_n$ and $y_1,\ldots,y_n$ are orthonormal and the action of $\SU_n$ leaves lengths invariant. 
\end{Rem}

Finally, we analyze the spherical amoebas, i.e. the images $\sLog_{\Gamma, t}(G/H)$ of $G/H$ under the spherical logarithm maps, and explicitly compute their limit as $t \to 0$. In our case of $G/H = \SL_n(\C)/\SL_{n-1}(\C)$ we can prove that the limit is precisely equal to the valuation cone, see Proposition~\ref{prop: limit in SLn SLn-1 case}. Recall that our spherical logarithm map is given by 
\[
  \sLog_{\Gamma, t} \colon \SL_n(\C)/\SL_{n-1}(\C) \to \mathcal{Q}_{\SL_n/\SL_{n-1}}\cong \R^2; (x,y) \mapsto 
  (\log_t \phi_{\omega_{n-1}}, \log_t \phi_{\omega_1}).
  \]
From~\eqref{eq: sphericals for SLn/SLn-1} we can also express this more explicitly as 
\[
\sLog_{\Gamma, t}(x,y) = ( \log_t(|x_1|^2 + \ldots + |x_n|^2), \log_t(|y_1|^2 + \ldots + |y_n|^2)) \text{.}
\]
In this special case, an explicit computation yields the following.

\begin{Prop}\label{prop: limit in SLn SLn-1 case} 
Let $G=\SL_n(\C)$ and $H=\SL_{n-1}(\C)$. Then 
$\sLog_{\Gamma, t}G/H) = \mathcal{V}_{G/H}$ for all $t<1$. 
In particular,  as $t$ approaches $0$, the sets $\sLog_{\Gamma, t}(G/H)$ converge to the valuation cone $ \overline{\mathcal{V}}_{G/H}$ in the sense of Kuratowski. 
\end{Prop}

\begin{proof} 
Let $(\xx,\yy) \in \SL_n(\C)/\SL_{n-1}(\C)$, i.e., $\sum_{i=1}^n x_i y_i = 1$. 
Let $t \in \R$ with $0 < t < 1$. Then $\mathrm{ln}(t) <0$ which implies that $\log_t(x) = \frac{\ln(x)}{\ln(t)}$ is a strictly decreasing function on $\R_{>0}$. Using this, we can make the following computation: 
  \begin{multline*}
    0 = \log_t |1| = \log_t |\sum_{i=1}^n x_i y_i | \ge \log_t \rleft( \sqrt{\sum_{i=1}^n |x_i|^2} \cdot \sqrt{\sum_{i=1}^n |y_i|^2} \rright) \\
    = \log_t\rleft( \sqrt{\sum_{i=1}^n |x_i|^2} \rright) + \log_t \rleft( \sqrt{\sum_{i=1}^n |y_i|^2} \rright) = \tfrac{1}{2} \log_t(\phi_{\omega_{n-1}}(\xx,\yy)) + \tfrac{1}{2} \log_t(\phi_{\omega_1}(\xx,\yy))
  \end{multline*}
  where the first inequality uses the standard Cauchy-Schwartz inequality together with the fact that $\log_t$ is a 
  decreasing function. It follows that $\sLog_{\Gamma, t}(\SL_n/\SL_{n-1}) \subseteq \mathcal{V}_{\SL_n/\SL_{n-1}}$.
It remains to show that $\mathcal{V}_{\SL_n/\SL_{n-1}} \subseteq \sLog_{\Gamma, t}(\SL_n(\C)/\SL_{n-1}(\C)))$. Let $(a,b) \in \mathcal{V}_{G/H}$, i.e., $(a,b) \in \R^2$ with $a+b \le 0$. Then $b \le -a$, and, as $x \mapsto t^x$ is a strictly decreasing function for $0 < t < 1$, we get that $t^b \ge t^{-a}$ or equivalently $t^b-t^{-a} \ge0$. Hence $y = \sqrt{t^b-t^{-a}} \in \R_{\ge0}$ is defined. We set $\xx = (t^{a/2},0,\ldots,0)$ and $\yy=(t^{-a/2},y,0,\ldots,0)$ and observe that $(\xx,\yy) \in \SL_n(\C)/\SL_{n-1}(\C)$. Finally, we may compute that
  \[
  \sLog_{\Gamma, t}(\xx,\yy) = \log_t(|t^{a/2}|^2,|t^{-a/2}|^2 + |\sqrt{t^b-t^{-a}}|^2) = (\log_t(t^a),\log_t(t^{-a}+t^b-t^{-a}))=(a,b) 
  \]
  so $(a,b) \in \sLog_{\Gamma, t}(\SL_n(\C)/\SL_{n-1}(\C))$ as desired. 
  \end{proof} 
  
\subsection{\texorpdfstring{The space of hyperbolic triangles}{The space (SL2 x SL2 x SL2)/SL2}}
\label{subsec: hyperbolic triangles}

In this section we consider the group $G = \SL_2(\C) \times \SL_2(\C) \times \SL_2(\C)$ 
with the diagonally embedded subgroup $H = \SL_2(\C)_{\diag} \subseteq G$. 
Let $B_G \subseteq G$ denote a fixed choice of Borel subgroup of $G$ given by $B_G = B^- \times B^- \times B^-$ where $B^-$ denotes the lower-triangular matrices in $\SL_2(\C)$. Let $T_G \subseteq B_G$ denote the maximal torus of $G$ given by triples of diagonal matrices in $\SL_2(\C)$. We denote by $\omega_i$ for $i = 1, 2, 3$, the pullback to $B_G$ of the fundamental weight of $\SL_2(\C)$ via the projection of $B_G = B^- \times B^- \times B^-$ to its $i$-th factor. 

To see that $G/H = \SL_2(\C) \times \SL_2(\C) \times \SL_2(\C)/\SL_2(\C)$ is spherical, it suffices to find an open dense $B$-orbit. To see this it would suffice to see that the double quotient $B \setminus G/H$ is finite. 
Since $B \setminus G = \mathbb{P}^1 \times \mathbb{P}^1 \times \mathbb{P}^1$ and $\SL_2(\C)$ acts transitively on triples of pairwise distinct points in $\mathbb{P}^1$, it follows that there are finitely many $\SL_2(\C)$-orbits in $B \setminus G$ and we are done.
Moreover, since $H = \SL_2(\C)$ is affine, $G/H$ is affine \cite[Theorem 3.8]{Timashev} and hence quasi-affine. Thus we are in the setting of Section~\ref{sec: limit}. 

Next we describe the colors of $G/H$. From the previous paragraph we can see that the complement of the open dense $B$-orbit consists of equivalence classes of triples of points where at least two of the points coincide. Below we describe these colors precisely in terms of the coordinates on the three copies of $\SL_2(\C)$. Specifically, we claim that the set $\mathcal{D}$ of colors of $G/H$, i.e. the set of $B_G$-invariant prime divisors of $G/H$, is exactly $\mathcal{D} = \{D_{12}, D_{13}, D_{23}\}$ where 
\[
D_{ij} := \{ \det(A_{ij}) = 0 \}/\SL_2(\C)_{\diag} \subseteq \SL_2(\C) \times \SL_2(\C) \times \SL_2(\C)/\SL_2(\C)_{\diag}
\]
and $A_{ij}$ is the $2 \times 2$ matrix obtained from a triple of matrices
\begin{equation}\label{eq: triple of matrices}
\left( \begin{bmatrix} x_{11} & x_{12} \\ x_{21} & x_{22} \end{bmatrix}, 
\begin{bmatrix} y_{11} & y_{12} \\ y_{21} & y_{22} \end{bmatrix}, 
\begin{bmatrix} z_{11} & z_{12} \\ z_{21} & z_{22} \end{bmatrix} \right) 
\end{equation}
by taking the top rows of the $i$-th matrix and the $j$-th matrix. Thus, for example, $A_{12} = \begin{bmatrix} x_{11} & x_{12} \\ y_{11} & y_{12} \end{bmatrix}$ so $\det(A_{12}) = x_{11} y_{12} - y_{11} x_{12}$. To see that $\mathcal{D}$ is the set of colors, it would suffice to show that $\det(A_{ij})$ is a $B_G$-semi-invariant for all $i, j$ (so its vanishing locus is $B_G$-invariant), that each $D_{ij}$ is prime, and that the union $D_{12} \cup D_{13} \cup D_{23}$ is precisely the complement of the open dense $B_G$-orbit of $\left( \begin{bmatrix} 1 & 0 \\ 0 & 1 \end{bmatrix}, \begin{bmatrix} 0 & -1 \\ 1 & 0 \end{bmatrix}, \begin{bmatrix} 1 & -1 \\ 0 & 1 \end{bmatrix} \right) \cdot \SL_2(\C)_{\diag}$. It is a simple computation to see that $\det(A_{ij})$ is a $B_G$-semi-invariant with associated weight $\Omega_{ij} = \omega_i + \omega_j$, so the first claim is proved. 
To see that the union of the $D_{ij}$ is the complement of the open $B_G$-orbit, note that any element in this $B_G$-orbit must be of the form 
$\left( \begin{bmatrix} a_1 & 0 \\ b_1& c_1 \end{bmatrix}, \begin{bmatrix} 0 & -a_2 \\ c_2 & -b_2 \end{bmatrix}, \begin{bmatrix} a_3 & -a_3 \\ b_3 & -b_3 + c_3\end{bmatrix} \right)$ for some $a_i, b_i, c_i \in \C$ with $a_i c_i = 1$ for all $i$, up to the action of $\SL_2(\C)_{\diag}$ on the right.  Thus to prove the claim it suffices to see that for any three pairwise linearly independent vectors in $\C^2$, there exists a basis with respect to which the coordinates of these vectors are of the form $(a_1, 0), (0, -a_2), (a_3, -a_3)$ for some non-zero $a_i$, but this is elementary linear algebra. 
Finally we must show that each $D_{ij}$ is prime. To show this, it would suffice to see that $D_{ij}$ contains an open dense $B_G$-orbit, since $B_G$ is irreducible. We sketch the argument for $D_{12}$ since the others are similar. Consider the $B_G$-orbit of the point 
$\left( \begin{bmatrix} 1 & 0 \\ 0 & 1 \end{bmatrix}, \begin{bmatrix} 1 & 0 \\ 0 & 1 \end{bmatrix}, \begin{bmatrix} 1 & -1 \\ 0 & 1 \end{bmatrix} \right) \cdot \SL_2(\C)_{\diag}$. This $B_G$-orbit consists of the triples of matrices in $\SL_2(\C)$ (modulo $\SL_2(\C)_{\diag}$) 
as in~\eqref{eq: triple of matrices} satisfying $x_{12}=y_{12}=0$ and $z_{11} = - z_{12}$. We claim that this $B_G$-orbit is equal to $D_{12} \cap D_{13}^c \cap D_{23}^c$, which is open in $D_{12}$, where $D_{ij}^c$ 
denotes the complement of $D_{ij}$. To see this it suffices to see that any triple~\eqref{eq: triple of matrices} such that $(x_{11}, x_{12}), (y_{11}, y_{12})$ are linearly dependent, and $(x_{11},x_{12}), (z_{11}, z_{12})$ are linearly independent, and $(y_{11}, y_{12}), (z_{11}, z_{12})$ are linearly independent, can be transformed by multiplication by the (same) element of $\SL_2(\C)$ to a triple satisfying the conditions $x_{12}=y_{12}=0$ and $z_{11} = - z_{12}$. This is elementary linear algebra. 
Thus $\mathcal{D}$ is the set of colors, as claimed. 

\begin{Lem}
The spherical homogeneous space $G/H$ is factorial, and the semigroup of highest weights $\Lambda^+_{G/H}$ of $\C[G/H]$ is generated by the weights of the colors $D_{12}, D_{13}, D_{13}$ in $\mathcal{D}$. 
\end{Lem} 

\begin{proof} 
Being factorial means by definition that every Weil divisor is a principal divisor, so 
$G/H$ is factorial if and only if its class group is $0$ . Now by smoothness of $G/H$ this is in turn equivalent to showing that $\mathrm{Pic}(G/H) = 0$ \cite[Ch. II, Cor. 6.16]{Hartshorne}.
As an algebraic variety, notice that $G/H$ is isomorphic to $\SL_2(\C) \times \SL_2(\C)$.  From \cite[Proposition 1]{Popov} we know that the Picard group $\mathrm{Pic}(\SL_2(\C) \times \SL_2(\C))$ is equal to $0$, so it is factorial. Hence $G/H$ is also factorial, and this proves the claim. 
   
To show the second claim, suppose $f \in \C[G/H]^{(B)}$ is a $B$-semi-invariant function on $G/H$. Then $\div(f)$ is $B$-invariant and hence a non-negative linear combination of $D_{12}, D_{13}, D_{23}$. This means $h:= f/A_{12}^a A_{13}^b A_{23}^c$ for some $a,b,c \in \Z_{\geq 0}$ is a nowhere-vanishing regular function on $G/H$. Pulling back to $G$ via the canonical projection $\pi: G \to G/H$ we obtain a nowhere-vanishing regular function on $G$. Such a function must be a character of $G$ by a result of Rosenlicht \cite{Rosenlicht}, and for $G=\SL_2(\C) \times \SL_2(\C) \times \SL_2(\C)$ must hence be a constant. Hence $f = A_{12}^a A_{13}^b A_{23}^c$ up to a constant, so its $B$-weight is in the non-negative span of $\Omega_{12}, \Omega_{13}, \Omega_{23}$ as desired. 
\end{proof}

Our main goal in this section is to answer Question~\ref{question: sph-amoeba-approach-sph-trop}(2) from Section~\ref{sec: limit} in the case currently under consideration, $G = \SL_2(\C) \times \SL_2(\C) \times \SL_2(\C)$ and $H=\SL_2(\C)_{\diag}$. 
More precisely, we aim to explicitly compute the image of the spherical logarithm map on $G/H$ for different real parameters $t > 0$ and to prove that this image converges in the Kuratowski sense to the valuation cone $\mathcal{V}_{G/H}$. To accomplish this, it will be useful to view the spherical functions $\phi_\lambda$ as being defined on the double coset space $K \backslash G/H$, where $K$ is the maximal compact subgroup of $G$; this is valid since the spherical functions $\phi_\lambda$ are $K$-invariant. It will turn out that we can identify the space $K \backslash G/H$ with a space of hyperbolic triangles, as we now explain; this will then aid us in the computation.

We choose $K = \SU_2 \times \SU_2 \times \SU_2$ as our maximal compact subgroup of $G$.  
We first recall how to identify  $\SU_2 \backslash \SL_2(\C)$ with the hyperbolic $3$-space
\[
  \mathcal{H}^3 = \rleft\{ (z, r) \mid z \in \C, r \in \R_{\geq 0} \rright\} 
  \]
following the exposition of \cite{Hyperbolic}. 
Let $\mathbb{H} = \rleft\{ z+jr \mid z, r \in \C \rright\}$ be the algebra of quaternions, defined in the usual way. 
For $g = \begin{psmallmatrix} a & b \\ c & c \end{psmallmatrix} \in \SL_2(\C)$ and $P \in \mathbb{H}$, define an action of $\SL_2(\C)$ on $\mathbb{H}$ from the \textit{right} by
\begin{equation}  \label{equ-SL(2)-action}
  P^g = (dP - b)(a - cP)^{-1},
\end{equation}
where the operations are in the algebra of quaternions. This formula is analogous to the action of $\SL_2(\R)$ on the upper-half plane by M\"obius transformations. Note that since the quaternions are not commutative, the order of the multiplication in~\eqref{equ-SL(2)-action} is important. Now consider the subspace of the quaternions given by 
\[ 
  \mathcal{H}^3 := \{z+jr \mid z \in \C, r \in \R_{\geq 0}\} \subseteq \mathbb{H} 
\]
which we may identify with the hyperbolic $3$-space, via the identification $z+jr \mapsto (z,r)$. In this manner we will identify hyperbolic $3$-space with a subspace of $\mathbb{H}$. 

We now briefly recall some facts about $\mathcal{H}^3$, referring the reader to \cite{Hyperbolic} for details. The subspace $\mathcal{H}^3$ is invariant under the action of $\SL_2(\C)$ on $\mathbb{H}$. If we equip $\mathcal{H}^3$ with its standard hyperbolic metric $d \colon \mathcal{H}^3 \times \mathcal{H}^3 \to \R$, then $\SL_2(\C)$ acts by isometries on $\mathcal{H}^3$ and this gives an isomorphism between the group of orientation preserving isometries of $\mathcal{H}^3$ and $\mathrm{PSL}_2$ \cite[Proposition 1.3]{Hyperbolic}. Moreover, it is straightforward to check that the $\SL_2(\C)$-stabilizer of the point $0 + 1 \cdot j$ is $\SU_2$ \cite[Proposition 1.1]{Hyperbolic}. Finally, the group $\mathrm{PSL}_2$ acts doubly transitively on $\mathcal{H}^3$ in the following sense: for all $P,P',Q,Q' \in \mathcal{H}^3$ such that $d(P,P') = d(Q,Q')$ there exists $g \in \mathrm{PSL}_2$ such that $P^g = Q$ and $(P')^g = Q'$ \cite[Proposition 1.4]{Hyperbolic}. 

It follows that we can identify $\mathcal{H}^3$ with the space of cosets $\SU_2 \backslash \SL_2(\C)$, and the double coset space $K \backslash G/H$ corresponds to triples of points in the hyperbolic space modulo the action of orientation preserving isometries of $\mathcal{H}^3$. Thus we may identify $K \backslash G/H$ with \textbf{the space of hyperbolic triangles modulo the action of orientation preserving isometries of $\mathcal{H}^3$} (this is slight abuse of language, since we also allow extreme cases such as three collinear points).

Next we express the spherical functions $\phi_{\Omega_{ij}}$ in the language of hyperbolic triangles. First, to compute the spherical function $\phi_{\Omega_{12}}$, we must find a basis of the irreducible module $V_{12} = 
\rleft\langle G \cdot f_1 \rright\rangle$ where 
\[
  f_1 \coloneqq \det \begin{psmallmatrix} x_{11} & x_{12} \\ y_{11} & y_{12} \end{psmallmatrix}.
\]
A straightforward computation shows that the following regular functions form a basis of $V_{12}$: 
\[
  V_{12} = \C  \underbrace{\det \begin{psmallmatrix} x_{11} & x_{12} \\ y_{11} & y_{12} \end{psmallmatrix}}_{\eqqcolon f_1} \oplus \C  \underbrace{\det \begin{psmallmatrix} x_{11} & x_{12} \\ y_{21} & y_{22} \end{psmallmatrix}}_{\eqqcolon f_2} \oplus \C  \underbrace{\det \begin{psmallmatrix} x_{21} & x_{22} \\ y_{11} & y_{12} \end{psmallmatrix}}_{\eqqcolon f_3} \oplus \C  \underbrace{\det \begin{psmallmatrix} x_{21} & x_{22} \\ y_{21} & y_{22} \end{psmallmatrix}}_{\eqqcolon f_4} \text{.}
\]
Let $h$ be the Hermitian inner product on $V_{12}$ associated to this basis; it is not difficult to check that $h$ is invariant under the action of $\SU_2 \times \SU_2$. It follows that the spherical function $\phi_{\Omega_{12}}$ is given by:
\[
  \phi_{\Omega_{12}} =  \frac{1}{2} \left( |f_1|^2 + |f_2|^2 + |f_3|^2 + |f_4|^2 \right)\text{.}
\]
We can compute the other two spherical functions in a similar manner. 

Now we wish to express these spherical functions in terms of the lengths of segments of hyperbolic triangles. From \cite[Proposition 1.7]{Hyperbolic} we know the hyperbolic distance $d( \cdot, \cdot)$ on $\mathcal{H}^3$ can be expressed explicitly as follows. Define $\delta(j, j^g) := \cosh d(j, j^g)$.  Suppose $g = \begin{psmallmatrix} a & b \\ c & d \end{psmallmatrix} \in \SL_2(\C)$. Then
\[
  \delta(j, j^g) := \cosh d(j, j^g) = \tfrac{1}{2} ( |a|^2 + |b|^2 + |c|^2 + |d|^2) \text{.}
\]
Now let $[(A,B,C)] \in K \backslash G /H$ be a hyperbolic triangle. By the above, the $\cosh$ of the length of the segment connecting the points $j^A$ and $j^B$ is given by
\[
  \delta (j^A,j^B) = \delta\rleft( j, j^{BA^{-1}} \rright) = \tfrac{1}{2} \rleft(|a|^2 + |b|^2 + |c|^2 + |d|^2\rright) \text{,}
\]
where the first equal sign follows by the fact that $\SL_2(\C)$ acts by isometries on $\mathcal{H}^3$ and where $BA^{-1} = \begin{psmallmatrix} a & b\\ c & d\end{psmallmatrix}$. On the other hand, we have
\begin{align*}
  \phi_{\Omega_{12}} (A,B,C) &= \phi_{\Omega_{12}} (I_2, BA^{-1},CA^{-1}) \\
  &= \frac{1}{2} \left( \rleft| \det \begin{psmallmatrix} 1 & 0\\ a & b\end{psmallmatrix}\rright|^2 + \rleft| \det \begin{psmallmatrix} 1 & 0\\ c & d\end{psmallmatrix}\rright|^2 + \rleft| \det \begin{psmallmatrix} 0 & 1\\ a & b\end{psmallmatrix}\rright|^2 + \rleft| \det \begin{psmallmatrix} 0 & 1\\ c & d\end{psmallmatrix}\rright|^2 \right)\\
  & = \tfrac{1}{2} \rleft(|a|^2 + |b|^2 + |c|^2 + |d|^2\rright) \text{.}
\end{align*}
Hence, we conclude 
\[
\phi_{\Omega_{12}}(A,B,C) = \delta(j^A, j^B) = \cosh d(j^A, j^B) 
\]
or in other words, the spherical function $\phi_{\Omega_{12}}$ coincides with the $\cosh$ of the length of the segment connecting the points $j^A$ and $j^B$. Similarly, we can see that the other spherical functions coincide with the $\cosh$ of the side lengths of the remaining segments of the hyperbolic triangle.

We now compute the image of $K \backslash G/H$ under the spherical logarithm map. From the above discussion, it follows that this image coincides with the image under the componentwise logarithm map
of the image of the following function, defined on the space of (possibly degenerate) hyperbolic triangles, which by slight abuse of notation we also denote by $\delta$: 
\[
 \delta \, \colon \,   \{ \text{hyperbolic triangle with side lengths} \; a,b,c\} \to \R^3; (a,b,c) \mapsto (\cosh(a), \cosh(b), \cosh(c))\text{.}
\]
We first compute the image of $\delta$. Recall that there exists a hyperbolic triangle with side lengths $a,b,c$ if and only if the sum of two side lengths is greater than or equal to the third side length. This can be equivalently reformulated as: $|b-c|\le a \le b+c$. Using the facts that $\cosh(x)$ is strictly increasing on $[0,\infty)$ and is an even function, it is straightforward to verify that this is equivalent to
\[
  (\cosh(a) - \cosh(b-c)) \cdot (\cosh(b+c) - \cosh(a))\ge0
\]
By using the addition theorem of $\cosh(x)$ and the equation $(\sinh(x))^2=(\cosh(x))^2-1$, this can be reformulated as follows: 
\[
  1 + 2 \cosh(a)\cosh(b)\cosh(c)-(\cosh(a))^2-(\cosh(b))^2-(\cosh(c))^2\ge0\text{.}
\]
Finally, we recall that the image of $f(x) = \cosh(x)$ is $[1,\infty)$. Putting these observations together yields the following result.

\begin{Prop}\label{prop:image-hyper}
  The image of the spherical logarithm map $\sLog_t \colon G/H \to \R^3$ is equal to 
  \[
    \rleft\{ \rleft( \log_t(x), \log_t(y), \log_t(z) \rright) : 1+2xyz-x^2-y^2-z^2\ge0 \textup{ and } x \geq 1, y \geq 1, z \geq 1 \rright\}
      \]

\end{Prop}

We must also compute the valuation cone $\mathcal{V}$. We use the language of minimal parabolics $P_\alpha$ which \emph{move} a color. More precisely, following \cite[Section 3.10]{Timashev} we say that a minimal parabolic $P_\alpha$ moves a color $D$ if $P_\alpha \cdot D \neq D$. In our case of $G = \SL_2(\C) \times \SL_2(\C) \times \SL_2(\C)$ 
with the diagonally embedded subgroup $H = \SL_2(\C)_{\diag} \subseteq G$, let us denote the simple roots (induced by our choice of $B_G$ and $T_G$ above) of the three factors by $\alpha$, $\beta$ and $\gamma$ respectively. In order to find the spherical roots, following \cite[Definition 30.21]{Timashev} and the surrounding discussion in \cite[Section 3.10]{Timashev} we compute how the minimal parabolics $P_\alpha, P_\beta, P_\gamma$ move the colors $D_{12}, D_{13}, D_{23}$. Here $P_\alpha$ is the minimal parabolic containing $B_G$ and the corresponding $\alpha$, and similarly for the othrs. In fact it is straightforward to show that $P_\alpha$, moves the two colors $D_{12}$ and $D_{13}$. Similarly, we can compute that the minimal parabolic $P_\beta$ moves the colors $D_{12}$, $D_{23}$ and the minimal parabolic $P_\gamma$ moves the colors $D_{13}$, $D_{23}$. Thus the roots $\alpha, \beta$ and $\gamma$ are spherical 
(in the language of \cite{Timashev}, all colors are of type $a$) and therefore we obtain 
\[
  \mathcal{V} = \{ v \in \mathcal{Q}_{G/H} : \langle v, \alpha \rangle \le 0, \langle v,\beta\rangle\le0, \langle v,\gamma\rangle\le0\}\text{.}
\]
If we identify the weight lattice $\Lambda$ with $\Z^3$ via the basis $\Omega_{12}, \Omega_{13}, \Omega_{23}$, then with respect to the corresponding dual basis, the valuation cone has primitive ray generators:
\[
  \begin{psmallmatrix}-1\\-1\\\hphantom{-}0\end{psmallmatrix}, \begin{psmallmatrix}-1\\\hphantom{-}0\\-1\end{psmallmatrix}, \begin{psmallmatrix}\hphantom{-}0\\-1\\-1\end{psmallmatrix}\text{.}
\]
Hence the valuation cone $\mathcal{V}$ is in this case entirely contained in the negative orthant, generated as a cone by the non-negative linear combinations of the above three generators. 

We now claim that the limit of the images of $\sLog_t$, in this case, is equal to the valuation cone $\mathcal{V}$. 

First note that as $t \to 0$ the image of the spherical logarithm map will be entirely contained in the negative orthant $\R^3_{\leq 0}$ due to the condition that $x, y, z \geq 1$ in Proposition~\ref{prop:image-hyper}. Second, recall that by Corollary~\ref{corollary Kuratowski} we already know that the Kuratowski limit of the images $\sLog_t(G/H)$ must contain the valuation cone. To show that the limit is equal to the valuation cone, we must therefore show that any point in the Kuratowski limit must be contained in the valuation cone. Indeed, considering the inequality in Proposition~\ref{prop:image-hyper}, it is not hard to see that if $(a,b,c) \in \R^3_{<0}$ is in the Kuratowski limit then we must have that 
\[
1+2t^{a+b+c} - t^{2a} - t^{2b} - t^{2c} \geq 0
\]
for $t$ small enough. From this it follows that 
$\lvert a \rvert - \lvert b \rvert + \lvert c \rvert \geq 0$ and $\lvert a \rvert + \lvert b \rvert - \lvert c \rvert \geq 0$ and  $ - \lvert a \rvert + \lvert b \rvert + \lvert c \rvert \geq 0$, which means that the $(a,b,c)$ lies in the valuation cone. Hence the Kuratowski limit coincides with the valuation cone, as desired.


\end{document}